\documentclass[pdflatex,sn-mathphys-num]{sn-jnl}

\usepackage{graphicx}%
\usepackage{multirow}%
\usepackage{amsmath,amssymb,amsfonts}%
\usepackage{amsthm}%
\usepackage{mathrsfs}%
\usepackage[title]{appendix}%
\usepackage{xcolor}%
\usepackage{textcomp}%
\usepackage{manyfoot}%
\usepackage{booktabs}%
\usepackage{algorithm}%
\usepackage{algorithmicx}%
\usepackage{algpseudocode}%
\usepackage{listings}%
\usepackage{enumitem}
\usepackage{lscape}
\usepackage{url}
\usepackage{hyperref}
\usepackage{longtable}

\newtheorem{theorem}{Theorem}
\newtheorem{lemma}[theorem]{Lemma}
\newtheorem{assumption}[theorem]{Assumption}

\newcommand{\dist}{{\rm dist}}
\DeclareMathOperator*{\argmin}{arg\,min}

\newcommand{\red}{\textcolor{red}{}\textcolor{black}}

\begin{document}

\title[Article Title]{A stochastic moving ball approximation method for smooth convex  constrained minimization }


\author[1]{\fnm{Nitesh Kumar} \sur{Singh}}\email{nitesh.nitesh@stud.acs.upb.ro}
\equalcont{These authors contributed equally to this work.}

\author*[1,2]{\fnm{Ion} \sur{Necoara}}\email{ion.necoara@upb.ro}
\equalcont{These authors contributed equally to this work.}


\affil[1]{\orgdiv{Automatic Control and Systems Engineering Department}, \orgname{National University of Science and Technology Politehnica Bucharest}, \orgaddress{\street{Spl. Independentei 313}, \city{Bucharest}, \postcode{060042}, \state{Bucharest}, \country{Romania}}}

\affil[2]{\orgname{Gheorghe Mihoc-Caius Iacob Institute of Mathematical Statistics and Applied Mathematics of the  Romanian Academy}, \orgaddress{\street{Street}, \city{Bucharest}, \postcode{050711}, \state{Bucharest}, \country{Romania}}}



\abstract{In this paper, we consider constrained optimization problems with convex objective and smooth convex  constraints. We propose a new stochastic gradient algorithm, called the Stochastic Moving Ball Approximation (SMBA) method, to solve this class of problems, where at each iteration we first take a (sub)gradient step for the objective function and then perform a projection step onto one ball approximation of a randomly chosen constraint. The computational simplicity of SMBA, which uses first-order information and considers only one constraint at a time, makes it suitable for large-scale problems with many functional constraints.  We provide a convergence analysis for the SMBA algorithm using basic assumptions on the problem, that yields new convergence rates in both optimality and feasibility criteria evaluated at some average point. Our convergence proofs are novel since we need to deal properly with infeasible iterates and with quadratic upper approximations of constraints that may yield empty balls. We derive convergence rates of order   $\mathcal{O} (k^{-1/2})$  when the objective function is convex, and $\mathcal{O} (k^{-1})$ when the objective function is strongly convex. Preliminary numerical experiments on quadratically constrained quadratic problems demonstrate the viability and performance of our method when compared to some existing state-of-the-art optimization methods and software.}




\maketitle


\section{Introduction}\label{sec1}

In this paper, we address the intricate challenges associated with the optimization of smooth functions subject to complex constraints. The central focus is on the following constrained optimization problem:
\begin{equation}\label{eq:probsmooth}
	\begin{array}{rl}
	f^* = & \min\limits_{x \in \mathcal{Y} \subseteq \mathbb{R}^n} \; f(x) \\
	& \text{subject to } \;  h(x,\xi) \le 0 \;\; \forall \xi \in \Omega,
	\end{array}
\end{equation}

\noindent where the constraints \red{$h(\cdot, \xi): \mathbb{R}^n \rightarrow \mathbb{R},\; \forall \xi \in \Omega$}, are assumed to be proper, convex, continuously differentiable smooth functions, while the objective \red{$f: \mathbb{R}^n \rightarrow \mathbb{R}$} is a proper convex function.  Additionally, $\mathcal{Y}$ is a nonempty closed convex set {which admits easy projections} and $\Omega$ is an arbitrary collection of indices (can be even infinite).  {Hence, we separate the feasible set in two parts:  one set, $\mathcal{Y}$, admits easy projections and the other part is not easy for projection as it is described by the level sets of some convex functions $h(\cdot,\xi)$'s.}  For such problems, we face many computational and theoretical challenges,  particularly, when they involve a \textit{large} number of constraints. 

\medskip 

\noindent \textit{Motivation}.  General constrained problems in the form of \eqref{eq:probsmooth}, with a large number of constraints, appear in many applications, such as systems and control \cite{BerBac:22, NedDin:14, NecNed:21},  machine learning and data science \cite{Vap:98, CheLi:21, BhaGra:04}, signal processing \cite{Nec:20, Tib:11}, and operations research and finance \cite{RocUry:00}. For example, in chance constrained minimization one needs to solve a problem of the form \cite{JinWan:22}: 
\begin{align}
\label{eq:chance_con_prob}
    \min_{x \in \mathcal{Y}} f(x) \quad  \text{s.t.} \;\;\;  \texttt{Prob$\red{(h(x,\xi)}\le 0)$}\ge 1 - \tau,
\end{align}
where $\xi$ is a random variable defined over a probability space. It has  been proven in \cite{CamGar:11} that the feasible set $\{ x \in \mathcal{Y} : \red{h(x,\xi)}\le 0 \quad  \forall \xi \in \Omega \}$
approximates the true feasible set of \eqref{eq:chance_con_prob} as the number of indices in $\Omega$ is sufficiently large. Therefore, chance constrained minimization \eqref{eq:chance_con_prob} can be well approximated by an optimization problem of the form  \eqref{eq:probsmooth}. Moreover,  quadratic functions are a powerful modeling construct in optimization and appear in various fields such as statistics, machine learning, portfolio optimization, optimal power flow, and control theory \cite{CheLi:21,WanTro:21}. In all these applications one needs to solve quadratically constrained quadratic programs (QCQP) of the form: 
\begin{align}
\label{qcqp}
& \min\limits_{x \in \mathcal{Y} \subseteq \mathbb{R}^n} \; \frac{1}{2}x^TQ_f x + q_f^T x \quad  \text{s.t.} \;\;\;  
	   \frac{1}{2} x^T Q_i x + q_i^T x - b_i \le 0 \;\; \forall i = 1:m, 
\end{align}
\noindent which is a particular case of our optimization problem \eqref{eq:probsmooth}, when the matrices $Q_f$ and $Q_i$, for all $i = 1:m$, are positive semidefinite.

\medskip

\noindent \textit{Related work}. 
In the realm of optimization, various first-order methods were designed to address the smooth optimization problem \eqref{eq:probsmooth}. A notable algorithm, which also represents the motivation for our current work,  is the moving balls approximation (MBA) method  \cite{AusTeb:10}. Specifically tailored for smooth (possibly nonconvex) constrained problems, MBA  generates a sequence of feasible points that approximates the constraints set by a sequence of balls, thus converting the original problem into a sequence of convex subproblems. Despite the fact that MBA exhibits linear convergence for smooth and strongly convex objective and constraints and asymptotic convergence for the nonconvex case, it requires \textit{initialization} at a feasible point and uses at each iteration the \textit{full} set of constraints, thus making the subproblem difficult to solve when the number of constraints is large. \red{To obtain a simpler subproblem,  in the same paper a variant of MBA is introduced that incorporates an active set technique, called MBA-AS, which has similar  convergence properties as original  MBA.} An enhanced version of MBA, named the sequential simple quadratic method (SSQM), is presented in \cite{Aus:13}. SSQM offers similar convergence results to MBA but eliminates the necessity to initiate from a feasible point. Other extensions of the MBA method (with possibly unbiased estimates of the gradients) are provided in \cite{BooLan:22},\cite{BooLan:23}, where composite objective and functional constraints are considered in the problem formulation and improved convergence results are derived in both convex and nonconvex settings. Finally, \cite{Nes:18} presents a first-order method to solve the smooth strongly convex problem  \eqref{eq:probsmooth} by converting it into a sequence of parametric max-type convex subproblems and derives a linear convergence rate for such scheme. However, these algorithms suffer from two technical issues: first, they need to consider the  \textit{full} set of constraints in the subproblem making them intractable when the number of constraints is large; second, they require an \textit{initialization} phase which can be computationally expensive.        

\medskip 

\noindent It is well-known that stochastic methods can overcome such technical problems. Stochastic (sub)gradient framework is a prevalent methodology for minimizing finite sum objective functions  \cite{RobMon:51, NemYud:83}. When combined with simple constraints, and the computation of the prox or projection operators is straightforward, a plethora of methods emerge, such as stochastic gradient descent (SGD) and stochastic proximal point (SPP) algorithms \cite{MouBac:11, NemJud:09, Nec:20, WanBer:16}.  Numerous extensions of SGD/SPP have been devised for solving convex problem \eqref{eq:probsmooth} with \textit{nonsmooth} functional constraints. At each iteration, such schemes perform a stochastic subgradient step for the objective and then project this step onto a halfspace lower approximation of a randomly chosen constraint,  see e.g., \cite{Ned:11, NecSin:22}, and their mini-batch variants \cite{SinNec:23, NecNed:21}.   Additionally,  Lagrangian primal-dual stochastic subgradient methods \red{(i.e., algorithms where both the primal and  dual variables are updated at each iteration)}  for convex nonsmooth functional constraint problems have been also designed, see e.g., \cite{Xu:20}. \red{The results from \cite{Xu:20} can  be also extended to the smooth constrained case using only tools from nonsmooth optimization}. All these algorithms achieve sublinear rates of convergence in terms of optimality and feasibility for an average sequence when they are combined with decreasing stepsize rules. However, extensions of stochastic first-order methods to problems with \textit{smooth} functional constraints are limited, mainly because for such constraints the corresponding quadratic \textit{upper} approximations may lead to empty sets (balls) in the subproblem. For example, a single-loop stochastic augmented Lagrangian primal-dual method is proposed in \cite{JinWan:22} for solving problem \eqref{eq:probsmooth} with smooth \red{nonconvex} functional constraints. This method achieves a convergence rate of order  $\mathcal{O} (k^{-1/4})$, where $k$ is the iteration counter, for finding an approximate KKT point. In cases where the initial point is feasible, this order reduces to $\mathcal{O} (k^{-1/3})$. However, to ensure the boundedness of dual multipliers, the penalty parameter in augmented Lagrangian and the stepsizes must be chosen dependent on the final iteration counter, hence it needs to be fixed a priori.

\medskip

\noindent \textit{Contributions}.  
In this paper we remove some of the previous issues, proposing a novel method for solving smooth constrained problems, called the \textit{Stochastic Moving Ball Approximation} (SMBA) algorithm, that performs first a {(sub)}gradient step for the objective and then projects this step onto a quadratic upper approximation of a randomly chosen constraint using an adaptive stepsize. Thus, SMBA updates only the primal decision variables, omitting dual updates, marking a distinctive contribution. \textit{To the best of our knowledge, this paper presents the first convergence analysis of such a primal stochastic algorithm in the smooth settings}. Hence, this paper makes the following key contributions.

 $(i)$ We consider a general optimization model in the form of an optimization problem \eqref{eq:probsmooth}, with {the objective} being either convex or strongly convex and with a large number of smooth convex functional constraints. 

 $(ii)$ We design a new primal algorithm, called Stochastic Moving Ball Approximation (SMBA), to solve such optimization problems with smooth, convex functional constraints. At each iteration, SMBA performs a {(sub)}gradient step for the objective function, and then takes a projection step onto a quadratic upper approximation of a randomly chosen constraint using an adaptive stepsize rule. The computational simplicity of the SMBA subproblem, which uses gradient information and considers only one constraint at a time, makes it suitable for problems with many functional constraints. Notably, SMBA updates only the primal decision variables and does not need to be initialized at a feasible point, which makes it advantageous over existing algorithms like MBA \cite{AusTeb:10} or the method in  \cite{JinWan:22}. Moreover, our algorithm is novel, since it needs to deal properly with quadratic upper approximations of constraints that may yield empty balls.   

$(iii)$  We provide a convergence analysis for the SMBA algorithm, under basic assumptions,  that yields new convergence rates in both optimality and feasibility criteria evaluated at some average point when the functional constraints are convex. In particular, we derive convergence rates of order   $\mathcal{O} (k^{-1/2})$  when the objective function is convex, and $\mathcal{O} (k^{-1})$ when the objective function is strongly convex. To the best of our knowledge,  these convergence results for SMBA are new and match the known convergence rates for stochastic first-order methods from the literature. Moreover, our proofs are novel, since we need to deal properly with infeasible iterates and with quadratic upper approximations of constraints that may yield empty sets. 


$(iv)$  Finally, numerical simulations on large-scale (strongly) convex quadratic problems with convex quadratic constraints using synthetic and real data demonstrate the viability and performance of our method when compared to some state-of-the-art optimization algorithms, such as its deterministic counterparts MBA \red{and MBA-AS} \cite{AusTeb:10}, and dedicated software, e.g., FICO \cite{FICO} and Gurobi~\cite{Gurobi}.

\medskip 

\noindent \textit{Content}.  The remainder of this paper is organized as \red{follows.} In Section 2, we introduce the essential notations and the key assumptions. Section 3, presents our new algorithm SMBA. In Section 4, we delve into the convergence analysis of the SMBA algorithm. Finally, in Section 5, we present numerical results that validate the practical efficacy of the SMBA algorithm.


\section{Notations and assumptions}
\noindent Throughout this paper, we use the following notations. For optimization problem \eqref{eq:probsmooth} we consider the individual sets $\mathcal{X}_\xi$ as:
\[ \red{\mathcal{X}_\xi = \{ x\in \mathbb{R}^n  : h(x, \xi) \le 0\}  \;\;\; \forall \xi \in \Omega .} \]
Thus,  the feasible set of \eqref{eq:probsmooth} is:
$$\mathcal{X}= \red{\mathcal{Y}\cap (\cap_{\xi \in \Omega} \mathcal{X}_\xi)}. $$
We assume the feasible set $\mathcal{X}$ to be non-empty {and to be difficult to project onto it}. We also assume that the optimization problem \eqref{eq:probsmooth} has a finite optimum and we let $f^*$ and $\mathcal{X}^*$ to denote the optimal value and the optimal set, respectively:
\[ f^* =  \min_{x\in \mathcal{X}} f(x), \quad \mathcal{X}^* = \{ x \in \mathcal{X}: \;  f(x) = f^* \} \neq \emptyset. \]
For any $x \in \mathbb{R}^n$ we denote its projection onto the optimal set $\mathcal{X}^* $ by $\bar{x}$, i.e., $\bar{x} = \Pi_{\mathcal{X}^*}(x)$ (we use the notation $\Pi_{\mathcal{B}}(x)$ for  projection of a point $x$ onto a closed convex set $\mathcal{B}$). For a given scalar $a$ we denote   $(a)_+ = \max (a,0)$. We make no assumption on the differentiability of $f$ and use,  with some abuse of notation, the same expression for the gradient or the subgradient of $f$ at $x$, that is $\nabla f(x) \in \partial f(x)$, where the subdifferential $\partial f(x)$ is either a singleton or a nonempty set. \red{Let $g$ be a smooth function (i.e., it has Lipschitz continuous gradient), then the quadratic upper approximation of $g$ at a point $x$ is denoted as:}
 \begin{align*}
 	q_{g}(y;x) = g(x) + \langle \nabla g(x), y-x \rangle + \frac{L_g}{2} \|y-x\|^2, 
 \end{align*}
where $L_g>0$ denotes the Lipschitz constant of the gradient of $g$.  Equivalently, we have the following rewriting for the quadratic approximation:
 \begin{align*}
 \frac{2}{L_{g}} q_{g}(y; x) & = \left\|y - \left(x - \frac{1}{L_g} \nabla g(x) \right) \right\|^2 - \frac{1}{L_g^2} \|\nabla g (x)\|^2 + \frac{2 g(x)}{L_g}\\
 & =  \|y-c_{x, g}\|^2 - R_{x, g}, 
 \end{align*}
 where we denote the center and the radius of the corresponding ball with 
 \[ c_{x, g} = x - \frac{1}{L_g} \nabla g(x) \;\; \text{and}\;\; R_{x, g} = \frac{1}{L_g^2} \|\nabla g (x)\|^2 - \frac{2 g(x)}{L_g}.\]

\noindent Finally, for optimization problem \eqref{eq:probsmooth}, we impose the following assumptions. First,  let us present the assumptions related to  the objective function~$f$:
\begin{assumption}\label{assumption1smooth}
    {The function  $f$ satisfies a (strong) convexity condition, i.e., there exists  $\mu \geq  0$ such that the following relation holds:}
	\begin{equation}
		\label{as:strg_conv_fsmooth} 
		f(y) \ge f(x) + \langle \nabla f(x), y-x \rangle + \frac{\mu}{2}  \|y  - x\|^2  \quad \forall x, y \in \mathcal{Y}.
	\end{equation}
\end{assumption}   
\noindent From Assumption \ref{assumption1smooth}, the objective function $f$ is convex when $\mu = 0$ and when $\mu > 0$ there exists a unique $x^* \in \mathcal{X}^*$ such that the following holds:
\begin{align}\label{eq:QGsmooth}
    f(x) - f(x^*) \ge \frac{\mu}{2}  \|x  - x^*\|^2 \quad \forall x \in \mathcal{X} \subseteq \mathcal{Y}.
\end{align}

\noindent Next, we state the assumptions for the functional constraints $h(\cdot, \xi)$:
\begin{assumption}\label{assumption2smooth}
The convex functional constraints $h(\cdot, \xi)$ are smooth (have Lipschitz continuous gradient), i.e., \red{there exists} $L_\xi > 0$ such that the following relation is true:
 \[  \| \nabla h(x, \xi)  - \nabla h(y, \xi)  \| \leq L_\xi \|x-y \|  \quad \forall x,y \in {\mathcal{Y}} \;\; \text{and} \;\; \forall \xi \in \Omega. \]
\end{assumption}
\noindent As a consequence of this assumption we also have (see Lemma 1.2.3 in \cite{Nes:18}):
\begin{align}\label{eq:smooth_hsmooth}
    h(y, \xi) \le h(x, \xi) + \langle \nabla h(x, \xi), y-x \rangle + \frac{L_\xi}{2} \|y-x\|^2 \;\;\;  \forall x,y \in {\mathcal{Y}},  \; \forall \xi \in \Omega.
\end{align}
\noindent Note that the right side of \eqref{eq:smooth_hsmooth} is the quadratic upper approximation for the smooth functional constraints. Finally, our last assumption is a linear regularity type condition on the functional constraints. Let us endow $\Omega$ with a probability distribution $\mathbf{P}$ and $\mathbb{E}_{\xi} \left[ \cdot \right]$ denotes the expectation of a scalar depending on the random variable $\xi \in \Omega$. 
\begin{assumption}\label{assumption3smooth}
    The functional constraints satisfy the following regularity condition, i.e., there exists a non-negative constant $c>0$ such that:
    \begin{equation}\label{eq:lin_regsmooth}
        \dist^2(y, \mathcal{X}) \le c \cdot  \mathbb{E}_{\xi} \left[ (h(y,\xi))_+^2 \right], \;\; \forall y \in {\mathcal{Y}}.
    \end{equation}
\end{assumption}

\noindent Note that Assumption \ref{assumption3smooth} is used frequently in the literature, see e.g.,  \cite{Ned:11,NecSin:22}.  It holds e.g., when the feasible set $\mathcal{X}$ has an interior point,  when the feasible set is polyhedral, or,  more general, when a strengthened Slater condition holds for the collection of functional constraints, such as the generalized Robinson condition \cite{LewPan:98}(Corollary 3).


\section{Stochastic moving ball approximation algorithm}
 In this section, we propose a novel stochastic gradient algorithm, called the Stochastic Moving Ball Approximation (SMBA) method, to solve problem \eqref{eq:probsmooth}.  At each iteration, SMBA takes a {(sub)}gradient step for the objective function, to minimize it, and then performs a projection step onto one ball approximation of a randomly chosen constraint to reduce the infeasibility. Hence, for a given iteration  $k$, we consider the random variable $\xi_k$ sampled from $\Omega$ according to the probability distribution $\mathbf{P}$ and then we perform the following steps (by convention we assume $0/0 = 0$):
\begin{center}
	\noindent\fbox{%
		\parbox{10.2cm}{%
			\textbf{Algorithm 1 (SMBA)}:\\
			$\text{Choose} \; x_0 \in \mathbb{R}^n, \; \text{positive stepsize sequence} \; \red{(\alpha_k)_{k\geq 0}} $  and $\beta > 0$. \\
			$\text{For} \; k \geq 0 \;  \text{repeat:}$
			\begin{align}
			& \text{Sample} \; \xi_k\sim \textbf{P} \; \text{and update:} \nonumber  \\
			&v_k = {\Pi_\mathcal{Y}}(x_{k} - \alpha_k \nabla f(x_{k})) \label{eq:alg2step1SMBA}\\
			& \text{Choose $p_{v_k, \xi_k}\ge 0$ and compute: } \nonumber\\
			& z_{k} = v_k - \beta \frac{(h(v_k, \xi_k))_+}{p_{v_k, \xi_k}} \nabla h(v_k, \xi_k) \label{eq:alg2step3SMBA}\\
			& x_{k+1} = \Pi_{\mathcal{Y}} (z_k).\label{eq:alg2stepSMBA}
			\end{align}
		}%
	}
\end{center} 

\medskip 

\noindent The computational simplicity of SMBA, which uses only the {(sub)}gradient of the objective and {the} gradient of a single constraint at a time, makes it suitable for large scale problems with many functional constraints. Note that many primal first-order algorithms \red{(i.e., algorithms updating at each iteration only the primal variables, $x$, based on gradient information)} have been proposed in the literature for solving optimization problems with {smooth} functional constraints \cite{AusTeb:10, Aus:13, BooLan:22, Nes:18}. However, these methods are usually deterministic, i.e., at each iteration \red{they need to evaluate the whole set of constraints and thus they do not scale well w.r.t. the number of constraints}. \textit{From our best knowledge,  SMBA is the first primal gradient algorithm that is of stochastic nature for solving optimization problems {of the form \eqref{eq:probsmooth}}, i.e., it uses the gradient of only one randomly chosen constraint at each iteration and thus it is scalable w.r.t. the number of functional constraints. }     

\medskip 

\noindent For simplicity of the exposition, for the functional constraint  $h(\cdot, \xi_k)$ we denote the center $c_{v_k, h(\cdot, \xi_k)} = c_{v_k, \xi_k}$ and the radius $R_{v_k, h(\cdot, \xi_k)} = R_{v_k, \xi_k}$.  In order to deal properly with  quadratic upper  approximations of constraints that may yield empty  balls, we consider the following \textit{novel choice} for $p_{v_k, \xi_k}$, which depends on the current iterate, hence it is adaptive:
\begin{align}\label{eq:alg2step2SMBA}
    p_{v_k, \xi_k} = \frac{(h(v_k, \xi_k))_+ \cdot L_{\xi_k}}{\left(1 - \frac{\sqrt{(R_{v_k, \xi_k})_+} }{ \|v_k - c_{v_k, \xi_k}\|}\right)}.
\end{align}
\noindent Our algorithm SMBA has three main steps.  First, {\eqref{eq:alg2step1SMBA} is a  projected (sub)}gradient step, i.e.,  we minimizes a quadratic approximation of the objective function $f$:
\[ {u_k} = \argmin\limits_{y \in \mathbb{R}^n} \; \left( \red{ q_f(y;x_k)} := f(x_k) + \langle \nabla f(x_k), y-x_k \rangle + \frac{1}{2\alpha_k} \|y-x_k\|^2 \right), \]
which yields:
\[ {u_k} = x_k - \alpha_k \nabla f(x_k), \]
{and then we project $u_k$ onto the simple set $\mathcal{Y}$ to obtain an intermittent point $v_k$, i.e., $v_k = \Pi_\mathcal{Y} (u_k)$. It is important to note that we perform the projections in steps \eqref{eq:alg2step1SMBA} and \eqref{eq:alg2stepSMBA} in order to have our iterates $v_k$ and $x_k$ in the set $\mathcal{Y}$ where our Assumptions \ref{assumption1smooth}-\ref{assumption3smooth} are valid.} Next, we have a stochastic gradient projection step minimizing the feasibility violation of the chosen random constraint $\xi_k$. Although our feasibility step has a compact expression, it describes three types of updates.  Indeed, combining the  expressions of $z_{k}$ from \eqref{eq:alg2step3SMBA}  with that of $p_{v_k, \xi_k}$ from \eqref{eq:alg2step2SMBA}, we get:
\begin{align*}
	z_{k} = \!\begin{cases} 
		v_k,  & \!\mbox{if }  h(v_k,\xi_k) \le 0  \\ 
		v_k - \frac{\beta}{L_{\xi_k}} \left(1 - \frac{\sqrt{R_{v_k, \xi_k}}}{\|v_k - c_{v_k, \xi_k} \|} \right) \nabla h(v_k, \xi_k),  & \!\mbox{if } R_{v_k,\xi_k} > 0, h(v_k,\xi_k) > 0\\
		v_k - \frac{\beta}{L_{\xi_k}} \nabla h(v_k,\xi_k), & \!\mbox{if } R_{v_k,\xi_k} \le 0, h(v_k,\xi_k) > 0 .
	\end{cases}
\end{align*}
Alternatively, if define the ball corresponding to the  quadratic upper approximation of the single functional constraint $h(\cdot,\xi_k)$: 
\[ \mathcal{B}_{{v}_k,\xi_k} = \left\{y: \; q_{h(\cdot,\xi_k)}(y;{v}_k) \le 0 \right\} = \left\{y: \; \|y-c_{{v}_k,\xi_k}\|^2 \le R_{{v}_k,\xi_k} \right\},   \]
then the update in \eqref{eq:alg2step3SMBA} can be written equivalently as (see Appendix A): 
\begin{align*}
& z_{k} = \begin{cases} 
v_k,  & \mbox{if }  h(v_k,\xi_k) \le 0  \\ 
(1-\beta)v_k + \beta  \Pi_{\mathcal{B}_{{v}_k,\xi_k}} (v_k), & \mbox{if } R_{v_k,\xi_k} > 0, \;  h(v_k,\xi_k) > 0\\
v_k - \frac{\beta}{L_{\xi_k}} \nabla h(v_k,\xi_k), & \mbox{if } R_{v_k,\xi_k} \le 0, \; h(v_k,\xi_k) > 0.
\end{cases}
\end{align*}
Thus, in this second step for a constraint $\xi_k\in \Omega$ if $v_k$ is feasible for that constraint we do not move. Otherwise,  we project the point $v_k$ onto the ball $\mathcal{B}_{v_k,\xi_k}$, which is nonempty if $R_{v_k,\xi_k} > 0$; otherwise, if $R_{v_k,\xi_k} \leq  0$, then the ball is empty and we choose to minimize the quadratic approximation of $h(\cdot,\xi_k)$ at $v_k$, denoted $q_{h(\cdot,\xi_k)}(y;v_k)$. Finally, in the third step, we project $z_k$ onto the simple set $\mathcal{Y}$ (by simple we mean that the projection onto this set is easy). 

\medskip 

\noindent Next, we provide a lemma which shows that our choice for $p_{v_k, \xi_k}$ is well-defined.

\medskip 
\begin{lemma}\label{lemma_p_defSMBA}
Let Assumption \ref{assumption2smooth} hold. Then, the parameter 	$p_{v_k, \xi_k}$ from \eqref{eq:alg2step2SMBA}  is well-defined and satisfies:
\begin{align}\label{eq:p_defSMBA}
	p_{v_k, \xi_k} \begin{cases}
		= 0, & \text{ if } \; h(v_k, \xi_k) \le 0,\\
		> 0, & \text{ if } \; h(v_k, \xi_k) > 0.
	\end{cases}
	\end{align} 
\end{lemma}

\begin{proof}
	We will prove this statement by dividing it in two cases:\\
	\textit{Case (i):} When $h(v_k, \xi_k) \le 0 \implies (h(v_k, \xi_k))_+ = 0$, in this case we get: 
	\[ p_{v_k, \xi_k} \overset{\eqref{eq:alg2step2SMBA}}{=} 0. \]
	\textit{Case (ii):} When $h(v_k, \xi_k) > 0$, in this case $R_{v_k, \xi_k}$ can be positive, negative or zero. First, let us consider that $R_{v_k, \xi_k} > 0$, then we have:
	\begin{align}\label{eq:mid1SMBA}
		& 0 < h(v_k, \xi_k) \implies 0 < \frac{2 h(v_k, \xi_k)}{L_{\xi_k}} \implies -\frac{2 h(v_k, \xi_k)}{L_{\xi_k}} < 0 \nonumber \\
		&  \implies \frac{\|\nabla h(v_k, \xi_k)\|^2}{L^2_{\xi_k}} - \frac{2 h(v_k, \xi_k)}{L_{\xi_k}} < \frac{\|\nabla h(v_k, \xi_k)\|^2}{L^2_{\xi_k}}\nonumber\\
		& \implies R_{v_k, \xi_k} < \|v_k - c_{v_k, \xi_k}\|^2 \implies \sqrt{ R_{v_k, \xi_k}} < \|v_k - c_{v_k, \xi_k}\|.
	\end{align}
which shows   that $p_{v_k, \xi_k}> 0$ in this case. Finally, considering the case  $R_{v_k, \xi_k} \le 0 \implies (R_{v_k, \xi_k})_+ = 0$, thus from the expression in \eqref{eq:alg2step2SMBA} we get, $p_{v_k, \xi_k}> 0$. Thus combining both cases completes the proof. 
\end{proof}

\noindent Next, we provide some useful bounds on $p_{v_k, \xi_k}$, which will play a key role in our convergence analysis below.

\begin{lemma}
\label{lemma_bound_pSMBA}
Let Assumption \ref{assumption2smooth} hold.	Then, the following statements hold. 
 \begin{enumerate}[label=\roman*.]
		\item [(i)] We always have the following lower bound for $p_{v_k, \xi_k}$:
		\begin{align}\label{eq:bound_p1SMBA}
			( h(v_k, \xi_k))_+ L_{\xi_k} \le p_{v_k, \xi_k}.
		\end{align}
		\item [(ii)] When $h(v_k, \xi_k) > 0$,  we have the following tighter bounds for $ p_{v_k, \xi_k}$:
		\begin{align}\label{eq:bound_p2SMBA}
			\frac{\|\nabla h(v_k, \xi_k) \|^2}{2} \le p_{v_k, \xi_k} \begin{cases}
				\le \|\nabla h(v_k, \xi_k) \|^2, & \text{ if } R_{v_k, \xi_k}>0,\\
				=  h(v_k, \xi_k) L_{\xi_k}, & \text{ if } R_{v_k, \xi_k}\le 0.
			\end{cases}
		\end{align}
	\end{enumerate}
\end{lemma} 

\begin{proof}
	From \eqref{eq:alg2step2SMBA} \red{and from the expressions of  $R_{v_k, \xi_k}$ and $ c_{v_k, \xi_k}$},  we obtain:
	\[ p_{v_k, \xi_k} = \frac{(h(v_k, \xi_k))_+  L_{\xi_k}}{\left(1 - \frac{\sqrt{(R_{v_k, \xi_k})_+} }{ \|v_k - c_{v_k, \xi_k}\|}\right)} \ge (h(v_k, \xi_k))_+  L_{\xi_k}. \]
Hence, this proves the first statement \eqref{eq:bound_p1SMBA}. 

\medskip 
 
\noindent  Next, to prove \eqref{eq:bound_p2SMBA}, we divide it into the following two cases:  \\
	\textit{Case (i):} When $R_{v_k, \xi_k} > 0$ and $ h(v_k, \xi_k) > 0$, we have:
	\begin{align*}
		p_{v_k, \xi_k} & \overset{\eqref{eq:alg2step2SMBA}}{=} \frac{h(v_k, \xi_k) \cdot L_{\xi_k}}{\left(1 - \frac{\sqrt{R_{v_k, \xi_k}} }{ \|v_k - c_{v_k, \xi_k}\|}\right)} = \frac{h(v_k, \xi_k) \cdot L_{\xi_k}}{\left(1 - \frac{\sqrt{R_{v_k, \xi_k}} }{ \|v_k - c_{v_k, \xi_k}\|}\right)} \frac{\left(1 + \frac{\sqrt{R_{v_k, \xi_k}} }{ \|v_k - c_{v_k, \xi_k}\|}\right)}{\left(1 + \frac{\sqrt{R_{v_k, \xi_k}} }{ \|v_k - c_{v_k, \xi_k}\|}\right)}\\
		& = \frac{h(v_k, \xi_k) \cdot L_{\xi_k}}{\left(1 - \frac{R_{v_k,\xi_k}}{ \|v_k - c_{v_k, \xi_k}\|^2}\right)}\left(1 + \frac{\sqrt{ R_{v_k, \xi_k}}}{ \|v_k - c_{v_k, \xi_k}\|}\right) \\
		& = \frac{h(v_k, \xi_k) \cdot L_{\xi_k}}{\left(1 - 1 + \frac{2 h(v_k, \xi_k) \cdot L_{\xi_k}}{ \|\nabla h(v_k, \xi_k) \|^2}\right)} \left(1 + \frac{\sqrt{ R_{v_k, \xi_k}}}{ \|v_k - c_{v_k, \xi_k}\|}\right)\\
		& = \frac{ \|\nabla h(v_k, \xi_k) \|^2}{2} \left(1 + \frac{\sqrt{ R_{v_k, \xi_k}}}{ \|v_k - c_{v_k, \xi_k}\|}\right).
	\end{align*}
	From \eqref{eq:mid1SMBA}, we know that $\sqrt{ R_{v_k, \xi_k}} < \|v_k - c_{v_k, \xi_k}\|$, when $h(v_k, \xi_k) > 0$ and $ R_{v_k, \xi_k} > 0$. Hence, we get:
	\[ \frac{\|\nabla h(v_k, \xi_k) \|^2}{2} \le p_{v_k, \xi_k}   \overset{\eqref{eq:mid1SMBA}}{\le} \|\nabla h(v_k, \xi_k) \|^2. \]
 
\noindent 	\textit{Case (ii):} When $R_{v_k, \xi_k} \le 0$ and $ h(v_k, \xi_k) > 0$,  we have:
	\[ R_{v_k, \xi_k} \le 0 \implies \frac{\|\nabla h(v_k, \xi_k) \|^2}{2} \le h(v_k, \xi_k) L_{\xi_k} \overset{\eqref{eq:alg2step2SMBA}}{=} p_{v_k, \xi_k}. \]
 	Hence, combining both cases, we obtain the desired bounds \eqref{eq:bound_p2SMBA}. 
\end{proof}

\noindent In our convergence analysis below we also assume the following bounds to hold: 
\begin{align} \label{eq:bound_xksmooth}
\begin{split}
    & \|\nabla f(\Pi_{\mathcal{X}} (x_k))\|\le \bar{B}_f, \;\|\nabla f(x_k)\|\le B_f,   \\
    & \|\nabla h(v_k, \xi)\|\le B_h, \; |h(v_k, \xi)| \le M_h \;\;\; \forall \xi \in \Omega, \;\;  \forall  k\ge 0.
\end{split}
\end{align} 
{In \eqref{eq:bound_xksmooth}, the upper bounds related to the objective function $f$ always hold provided that e.g., $f$ is Lipschitz continuous or is continuously differentiable function and the sequence $(x_k)_{k \geq 0}$ generated by algorithm  SMBA is bounded. Moreover, the upper bounds related to the functional constraints $h(\cdot, \xi)$,  with $\xi \in \Omega$, always hold when the sequence $(v_k)_{k \geq 0}$ generated by  SMBA is bounded (recall that $h(\cdot, \xi)$,  with $\xi \in \Omega$, are assumed smooth functions).}  Note that the boundedness of a sequence generated by some algorithm is frequently employed in optimization, see e.g., \red{Lemma $4.4$ in \cite{BolTeb:14} and Lemma $3.4$ in \cite{CohTeb:22}}. {In our settings however, since we know that the sequences $(x_k,v_k)_{k \geq 0}$ generated by SMBA are always in the set $\mathcal{Y}$,}  their boundedness can be automatically  ensured if the set $\mathcal{Y}$ in our problem \eqref{eq:probsmooth} is considered compact.


\section{Convergence analysis of SMBA} 
In this section, we provide a convergence analysis for the SMBA algorithm under the basic assumptions from Section $2$, yielding new global convergence rates in optimality and feasibility criteria. \textit{Our proofs are novel, since we need to deal properly with infeasible iterates and with quadratic upper approximations of constraints that yield empty sets.}  We start our convergence analysis with some preliminary lemmas that are crucial in the derivation of the convergence rates. First, we prove some feasibility relation between  $x_{k} $ and $v_k$.\\

\begin{lemma}\label{lemma_feasSMBA}
Let the bounds from \eqref{eq:bound_xksmooth} hold. Then, we have the following relation true for the sequences generated by SMBA algorithm:
	\begin{align}\label{eq:lem_feasSMBA}
          \frac{1}{2}\dist^2(x_{k}, \mathcal{X}) - \alpha_k^2 B_f^2 \le \dist^2(v_k, \mathcal{X})\quad \forall  k\ge 0.
	\end{align}
\end{lemma}

\begin{proof}
{From  the definition of projection (see \cite{NecSin:22}), we have:}
\[ {\|v - \Pi_\mathcal{X} (v) \|^2 \le \|v - y\|^2 \;\;\;  \forall v \in \mathbb{R}^n, \; y \in \mathcal{X}.}\]
{Hence, taking $v = x_k$ and $y = \Pi_{\mathcal{X}} (v_k) \in \mathcal{X}$}, we obtain:
\begin{align*}
    \dist^2(x_{k}, \mathcal{X}) & = \|x_{k} - \Pi_{\mathcal{X}} (x_{k})\|^2 \le \|x_{k} - \Pi_{\mathcal{X}} (v_k)\|^2 = \|x_{k} - {v_k + v_k} - \Pi_{\mathcal{X}} (v_k)\|^2 \\
    & \le  2 \|v_k - \Pi_{\mathcal{X}} (v_k) \|^2 + 2 {\|x_{k} - v_k\|^2 }\\
    & \overset{\eqref{eq:alg2step1SMBA}}{=} 2 \|v_k - \Pi_{\mathcal{X}} (v_k) \|^2 + 2 {\|\Pi_\mathcal{Y} (x_{k}) - \Pi_\mathcal{Y} (u_k)\|^2 } \\
    & \le 2\; \dist^2(v_k, \mathcal{X}) +  2\alpha_k^2 \|\nabla f(x_k)\|^2,
\end{align*}
{where the second inequality uses the relation $\| a + b \|^2 \le  2 \|a\|^2 + 2\|b\|^2$ for all $a,b \in \mathbb{R}^n$ and the last inequality uses that $x_k \in \mathcal{Y}, u_k = x_k - \alpha_k \nabla f(x_k)$ and the nonexpansiveness property of the projection operator, i.e., $\|\Pi_\mathcal{Y} (x_{k}) - \Pi_\mathcal{Y} (u_k)\|^2 \le \|x_k - u_k\|^2$}. After using the bound on the {(sub)}gradients of $f$ from \eqref{eq:bound_xksmooth} and rearranging the terms we obtain the claimed result.
\end{proof} 

\noindent Our next lemma establishes a relation between $ \|x_{k+1} - \bar{x}_{k+1}\|^2$ and  $\|v_k - \bar{v}_k\|^2$. \\ 

\begin{lemma}\label{lemma_main_reccSMBA}
Let Assumptions \ref{assumption2smooth} and \ref{assumption3smooth} hold and $h(\cdot, \xi)$, for all $\xi \in \Omega$, be convex functions. Additionally, assume that the bounds \eqref{eq:bound_xksmooth} are valid and \red{ $\beta \in (0,1)$}. Denote $ \mathcal{B}_h^2 = \max(B_h^2,  M_h \cdot \max_{\xi \in \Omega} L_\xi)$. Then, for any $k\ge 0$, we have the following recurrence true for sequences generated by SMBA:
\begin{align}\label{eq:main_reccSMBA}
    \mathbb{E}_{\xi_k}[\|x_{k+1} - \bar{x}_{k+1}\|^2] \le  \|v_k - \bar{v}_k\|^2 -  \frac{\beta  (1 - \beta)}{c \mathcal{B}_h^2} \dist^2 (x_k, \mathcal{X}) + \frac{2\beta (1 - \beta)}{ c \mathcal{B}_h^2}\alpha^2_k B_f^2.
\end{align}
\end{lemma}

\begin{proof} Recall that $\bar{x}$ denotes the projection of a vector $x$ onto the optimal set $\mathcal{X}^*$. Then, using the basic properties of the projection operator, we have:
\begin{align}\label{eq:dividingSMBA}
    & \|x_{k+1} - \bar{x}_{k+1}\|^2 \le \|\Pi_\mathcal{Y}(z_{k}) - \bar{v}_k\|^2  = \|\Pi_\mathcal{Y} (z_k)- \Pi_\mathcal{Y}(\bar{v}_k)\|^2 \le \|z_k - \bar{v}_k\|^2\nonumber\\
    & = \|z_k -v_k + v_k - \bar{v}_k\|^2 = \|v_k - \bar{v}_k\|^2 + 2 \langle z_k -v_k, v_k - \bar{v}_k \rangle + \|z_k -v_k \|^2\nonumber \\
    & \overset{\eqref{eq:alg2step3SMBA}}{=} \|v_k - \bar{v}_k\|^2 + 2 \beta \frac{ (h(v_k, \xi_k))_+}{p_{v_k, \xi_k}} \langle \nabla h(v_k, \xi_k), \bar{v}_k - v_k \rangle  + \| z_k -v_k \|^2 \nonumber \\
    & \le \|v_k - \bar{v}_k\|^2 + 2\beta \frac{ (h(v_k, \xi_k))_+}{p_{v_k, \xi_k}} \left( h(\bar{v}_k, \xi_k) - h(v_k, \xi_k)  \right) + \| z_k -v_k \|^2\nonumber\\
    & \overset{\eqref{eq:alg2step3SMBA}}{\le} \|v_k - \bar{v}_k\|^2 - 2\beta \frac{ (h(v_k, \xi_k))^2_+}{p_{v_k, \xi_k}}  + \beta^2 \frac{ (h(v_k, \xi_k))^2_+}{p^2_{v_k, \xi_k}}\| \nabla h(v_k, \xi_k)\|^2,
 \end{align}
where the last inequality follows from the fact that the constraints are feasible at $\bar{v}_k \in \mathcal{X}^* \subseteq \mathcal{X}$, i.e., $h(\bar{v}_k, \xi_k) \le 0$ and that $(a)_+ a = (a)_+^2$, for any scalar $a$. Furthermore, we proceed by considering the following \red{cases.}\\

\noindent \textit{Case (i):} When $h(v_{k}, \xi_k) \le 0 \implies (h(v_{k}, \xi_k))_+ = 0$. Then,  from \eqref{eq:p_defSMBA} we have $p_{ v_k, \xi_k} = 0$. Thus, the last two terms on the right-hand side of \eqref{eq:dividingSMBA} will disappear (recall that by convention $0/0 =0$). Hence, \eqref{eq:dividingSMBA} can be written as follows for any constant $\mathcal{B}_h >0$:
\begin{align}\label{eq:dividing_cases11SMBA}
    & \|x_{k+1} - \bar{x}_{k+1}\|^2 \le  \|v_k - \bar{v}_k\|^2 - \frac{ 2\beta (1 - \beta)}{\mathcal{B}_h^2} (h(v_k, \xi_k))_+^2 .
\end{align}

\noindent \textit{Case (ii):} When $h(v_k, \xi_k) > 0$ and $R_{v_k, \xi_k} > 0$, relation \eqref{eq:dividingSMBA} yields:
\begin{align}\label{eq:dividing_cases22SMBA}
    & \|x_{k+1} - \bar{x}_{k+1}\|^2 \overset{\eqref{eq:bound_p2SMBA}}{\le} \|v_k - \bar{v}_k\|^2 - 2\beta \frac{ (h(v_k, \xi_k))_+^2}{p_{ v_k, \xi_k}} + 2 \beta^2 \frac{ (h(v_k, \xi_k))_+^2}{p_{ v_k, \xi_k}} \nonumber\\
    & = \|v_k - \bar{v}_k\|^2  - 2\beta (1 - \beta)\frac{ (h(v_k, \xi_k))_+^2}{p_{ v_k, \xi_k}} \nonumber\\ 
    & \!\overset{\eqref{eq:bound_p2SMBA}}{\le} \|v_k - \bar{v}_k\|^2 - 2\beta (1 - \beta)\frac{ (h(v_k, \xi_k))_+^2}{\|\nabla h(v_k, \xi_k) \|^2} \nonumber\\ 
    &  \overset{\eqref{eq:bound_xksmooth}}{\le} \|v_k - \bar{v}_k\|^2 -  \frac{2\beta  (1 - \beta)}{B_h^2}(h(v_k, \xi_k))_+^2 .
\end{align}


\noindent \textit{Case (iii):} When $h(v_k, \xi_k) > 0$ and $R_{v_k, \xi_k} \le 0$, \red{from the definition of $R_{v_k, \xi_k}$} we have:
\[ p_{ v_k, \xi_k} = h(v_k, \xi_k) L_{\xi_k} \;\; \text{and } R_{v_k, \xi_k} \le 0 \implies \|\nabla h(v_k, \xi_k) \|^2 \le 2 h(v_k, \xi_k) L_{\xi_k}. \]
Using this in  \eqref{eq:dividingSMBA}, we get:
\begin{align}\label{eq:dividing_cases33SMBA}
    & \|x_{k+1} - \bar{x}_{k+1}\|^2 \le \|v_k - \bar{v}_k\|^2 - 2\beta \frac{ (h(v_k, \xi_k))_+^2}{p_{ v_k, \xi_k}} + 2\beta^2 \frac{ (h(v_k, \xi_k))_+^2}{p_{ v_k, \xi_k}} \nonumber\\
    & = \|v_k - \bar{v}_k\|^2 - 2\beta (1 - \beta)\frac{ (h(v_k, \xi_k))_+^2}{h(v_k, \xi_k) \cdot L_{\xi_k}} \nonumber\\ 
    &  \overset{\eqref{eq:bound_xksmooth}}{\le} \|v_k - \bar{v}_k\|^2 -  \frac{2\beta  (1 - \beta)}{M_h \cdot \max_{\xi \in \Omega} L_\xi}(h(v_k, \xi_k))_+^2. 
\end{align} 
Now, combining \eqref{eq:dividing_cases11SMBA}, \eqref{eq:dividing_cases22SMBA}  and \eqref{eq:dividing_cases33SMBA},  we get the following common recurrence:
\begin{align*} 
	& \|x_{k+1} - \bar{x}_{k+1}\|^2  \le  \|v_k - \bar{v}_k\|^2 -  \frac{2\beta  (1 - \beta)}{\mathcal{B}_h^2}(h(v_k, \xi_k))_+^2 ,
\end{align*}
where $ \mathcal{B}_h^2 = \max(B_h^2,  M_h \cdot \max_{\xi \in \Omega} L_\xi)$.  After taking expectation conditioned on $\xi_k$ and using the linear regularity condition from Assumption \ref{assumption3smooth}, we obtain:
\begin{align*} 
    & \mathbb{E}_{\xi_k}[\|x_{k+1} - \bar{x}_{k+1}\|^2] \le  \|v_k - \bar{v}_k\|^2 -  \frac{2\beta  (1 - \beta)}{\mathcal{B}_h^2}\mathbb{E}_{\xi_k}[(h(v_k, \xi_k))_+^2]\\
    & \overset{\eqref{eq:lin_regsmooth}}{\le} \|v_k - \bar{v}_k\|^2 -  \frac{2\beta  (1 - \beta)}{c \mathcal{B}_h^2} \dist^2 (v_k, \mathcal{X})\\
    & \overset{\eqref{eq:lem_feasSMBA}}{\le} \|v_k - \bar{v}_k\|^2 -  \frac{\beta  (1 - \beta)}{c \mathcal{B}_h^2} \dist^2 (x_k, \mathcal{X}) + \frac{2\beta (1 - \beta)}{ c \mathcal{B}_h^2}\alpha^2_k B_f^2.
 \end{align*}
This proves our statement. 
\end{proof}


\subsection{Convergence rates for SMBA under convex objective}
In this section, we consider that the objective function $f$ is convex, i.e., the Assumption \ref{assumption1smooth}$(ii)$ holds with $\mu  = 0$ and derive the convergence rates for SMBA. First, in the following lemma, we provide a main recurrence.\\

\begin{lemma}\label{lemma_vk_xkSMBA}
{Let Assumptions \ref{assumption1smooth} (with $\mu=0$), \ref{assumption2smooth} and \ref{assumption3smooth} hold and the bounds from \eqref{eq:bound_xksmooth} be valid. Moreover,  $h(\cdot, \xi)$ for all $\xi \in \Omega$ are convex functions.}  Further, let the stepsizes $\beta \in(0,1)$ and {$\alpha_k > 0$} for all $k \ge 0$. Then,  we have the following relation true for the sequences generated by SMBA algorithm:
\begin{align}\label{eq:xk_vkSMBA}
    \mathbb{E}[\|x_{k+1} - \bar{x}_{k+1}\|^2 \le& \mathbb{E}[\|x_k - \bar{x}_k\|^2] - 2 \alpha_k \mathbb{E}[\left( f(x_k) - f(\bar{x}_k) \right)] \\
    & \quad - \frac{\beta  (1 - \beta)}{ c \mathcal{B}_h^2}\mathbb{E}[\dist^2 (x_k, \mathcal{X})] + C_{\beta,c, \mathcal{B}_h} B_f^2 \alpha^2_k \quad k \ge 0, \nonumber
 \end{align}
where we denote $C_{\beta,c, \mathcal{B}_h} = \left( {1} + \frac{{2} \beta (1 - \beta)}{ c \mathcal{B}_h^2} \right)$.
\end{lemma}

\begin{proof}
From the basic properties of the projection {and recalling  that $\bar{x}_k \in \mathcal{X}^* \subseteq \mathcal{Y}$ and $u_k = x_k - \alpha_k \nabla f(x_k)$}, we get:
\begin{align*}
    & \|v_k  - \bar{v}_k\|^2 \le  \|v_k - \bar{x}_k\|^2 \; {\overset{\eqref{eq:alg2step1SMBA}}{=} \; \|\Pi_\mathcal{Y}(u_k)  - \Pi_\mathcal{Y}(\bar{x}_k)\|^2 \le \|u_k - \bar{x}_k \|^2} \\
    & = \|x_k - \bar{x}_k\|^2 - 2 \alpha_k    \langle \nabla f(x_k), x_k - \bar{x}_k \rangle + \alpha_k^2\|\nabla     f(x_k)\|^2\nonumber\\
    & \overset{\eqref{eq:bound_xksmooth}}{\le} \|x_k - \bar{x}_k\|^2 - 2 \alpha_k \left( f(x_k) - f(\bar{x}_k) \right) + B_f^2 \alpha_k^2,
\end{align*}
where the last inequality follows from the convexity of $f$. Now, using it in \eqref{eq:main_reccSMBA}, yields:
\begin{align*} 
    & \mathbb{E}_{\xi_k}[\|x_{k+1} - \bar{x}_{k+1}\|^2] \le  \|v_k - \bar{v}_k\|^2 -  \frac{\beta  (1 - \beta)}{c \mathcal{B}_h^2} \dist^2 (x_k, \mathcal{X}) + \frac{2\beta (1 - \beta)}{ c \mathcal{B}_h^2}\alpha^2_k B_f^2\\
    & \le \|x_k - \bar{x}_k\|^2 - 2 \alpha_k \left( f(x_k) - f(\bar{x}_k) \right) - \frac{\beta  (1 - \beta)}{ c \mathcal{B}_h^2}\dist^2 (x_k, \mathcal{X}) + \left( {1} + \frac{{2} \beta (1 - \beta)}{ c \mathcal{B}_h^2} \right) B_f^2 \alpha_k^2 \\
    & = \|x_k - \bar{x}_k\|^2 - 2 \alpha_k \left( f(x_k) - f(\bar{x}_k) \right) - \frac{\beta  (1 - \beta)}{ c \mathcal{B}_h^2}\dist^2 (x_k, \mathcal{X}) + C_{\beta,c, \mathcal{B}_h} B_f^2 \alpha^2_k,
 \end{align*}
where we denote $C_{\beta,c, \mathcal{B}_h} = \left( {1} + \frac{{2} \beta (1 - \beta)}{ c \mathcal{B}_h^2} \right)$. After taking full expectation we get the desired result. 
\end{proof}

\noindent Now we are ready to prove the main (sublinear) convergence rates of SMBA in the convex case. Let us define the following average sequences:
\[ \hat{x}_{k} = \frac{1}{S_k}  \sum_{t=0}^{k-1} \alpha_t x_{t},\;\; \text{ and }, \;\; \hat{w}_{k} = \frac{1}{S_k}  \sum_{t=0}^{k-1} \alpha_t \Pi_{\mathcal{X}} (x_{t}),  \]
where $S_k = \sum_{t=0}^{k-1}\alpha_t$ and note that $\hat{w}_k \in \mathcal{X}$, i.e., it is feasible.\\

\begin{theorem}\label{Th:convex_case2SMBA}
Let Assumptions \ref{assumption1smooth} {(with $\mu=0$)}, \ref{assumption2smooth} and \ref{assumption3smooth} hold and the bounds from \eqref{eq:bound_xksmooth} be valid. Moreover,  $h(\cdot, \xi)$ for all $\xi \in \Omega$ are convex functions. Then, choosing the stepsizes $\beta \in (0, 1)$ and {$\alpha_k > 0$} non-increasing, we get the following bounds  for the average sequence $\hat{x}_{k}$ in terms of optimality and feasibility violation:
\begin{align*}
    &\mathbb{E}[(f(\hat{x}_{k}) - f^*)] \le \frac{\|x_0 - \bar{x}_0\|^2}{2 S_k} + C_{\beta, c, \mathcal{B}_h}B_f^2 \frac{\sum_{t=0}^{k-1} \alpha_t^2}{2 S_k},\\
    & \mathbb{E}[\dist^2 (\hat{x}_{k}, \mathcal{X})] \le \frac{\alpha_0 c \mathcal{B}_h^2}{\beta  (1 - \beta)}\left( \frac{\|x_0 - \bar{x}_0\|^2}{S_k} + C_{\beta, C, \mathcal{B}_h} B_f^2 \frac{\sum_{t=0}^{k-1} \alpha_t^2}{S_k} \right).
\end{align*}
\end{theorem}

\begin{proof}
Since  $\alpha_k$ is a decreasing sequence, we have $\alpha_k/ \alpha_0 \le 1$ for all $k>0$.  Summing the relation \eqref{eq:xk_vkSMBA} from $0$ to $k-1$,  we get:
\begin{align*} 
    & \mathbb{E}[\|x_{k} - \bar{x}_{k}\|^2] \le \|x_0 - \bar{x}_0\|^2 - 2 \sum_{t=0}^{k-1} \alpha_t \mathbb{E}[(f(x_t) - f(\bar{x}_t))] \\
    & \quad  -  \frac{\beta  (1 - \beta)}{\alpha_0 c \mathcal{B}_h^2} \sum_{t=0}^{k-1}  \alpha_t \mathbb{E}[\| x_{t} - \Pi_{\mathcal{X}} (x_t)\|^2] + C_{\beta, c, \mathcal{B}_h} B_f^2 \sum_{t=0}^{k-1} \alpha_t^2.
\end{align*}
Now, divide the whole inequality by $S_k$, we have:
\begin{align*} 
    & \red{\frac{1}{S_k} \mathbb{E}[\|x_{k} - \bar{x}_{k}\|^2] \le \frac{1}{S_k} \|x_0 - \bar{x}_0\|^2 - \frac{2}{S_k} \sum_{t=0}^{k-1} \alpha_t \mathbb{E}[(f(x_t) - f(\bar{x}_t))]} \\
    & \red{\quad  -  \frac{\beta  (1 - \beta)}{\alpha_0 S_k c \mathcal{B}_h^2} \sum_{t=0}^{k-1}  \alpha_t \mathbb{E}[\| x_{t} - \Pi_{\mathcal{X}} (x_t)\|^2] + C_{\beta, c, \mathcal{B}_h} B_f^2\frac{\sum_{t=0}^{k-1} \alpha_t^2}{S_k} }\\
    & \red{= \frac{1}{S_k} \|x_0 - \bar{x}_0\|^2 - 2 \mathbb{E}\left[ \frac{\sum_{t=0}^{k-1} \alpha_t }{S_k} (f(x_t) - f^*)\right]}\\
    & \red{\quad -  \frac{\beta  (1 - \beta)}{\alpha_0 c \mathcal{B}_h^2} \mathbb{E} \left[\frac{\sum_{t=0}^{k-1} \alpha_t}{S_k}  \| x_{t} - \Pi_{\mathcal{X}} (x_t)\|^2 \right] + C_{\beta, c, \mathcal{B}_h} B_f^2\frac{\sum_{t=0}^{k-1} \alpha_t^2}{S_k} }\\
    & \red{\le \frac{1}{S_k} \|x_0 - \bar{x}_0\|^2 - 2 \mathbb{E}\left[ (f(\hat{x}_{k}) - f^*)\right]}\\
    & \red{\quad -  \frac{\beta  (1 - \beta)}{\alpha_0 c \mathcal{B}_h^2} \mathbb{E} \left[ \| \hat{x}_{k} - \hat{w}_{k}\|^2 \right] + C_{\beta, c, \mathcal{B}_h} B_f^2\frac{\sum_{t=0}^{k-1} \alpha_t^2}{S_k} },
\end{align*}
\red{where the first inequality uses the linearity of the expectation operator and the second inequality is derived from the definitions of $\hat{x}_k$ and $\hat{w}_k$ and also from the convexity of $f$ and $\|\cdot \|^2$, i.e., $\frac{\sum_{t=0}^{k-1} \alpha_t}{S_k}  \| x_{t} - \Pi_{\mathcal{X}} (x_t)\|^2 \ge \left\| \frac{\sum_{t=0}^{k-1} \alpha_t}{S_k}  (x_{t} - \Pi_{\mathcal{X}} (x_t)) \right\|^2 = \| \hat{x}_{k} - \hat{w}_{k}\|^2$. Now, using the fact that $\mathbb{E}[\|x_{k} - \bar{x}_{k}\|^2] \ge 0$, we further obtain: }
\begin{align*} 
    & 2 \mathbb{E}[(f(\hat{x}_{k}) - f^*)] + \frac{\beta  (1 - \beta)}{\alpha_0 c \mathcal{B}_h^2}\mathbb{E}[\|\hat{x}_{k} - \hat{w}_{k}\|^2 ]\\
    & \le \frac{\|x_0- \bar{x}_0\|^2}{S_k} + C_{\beta, c, \mathcal{B}_h} B_f^2 \frac{\sum_{t=0}^{k-1} \alpha_t^2}{S_k}.
\end{align*}
We get the following bound  in expectation for the average sequence $\hat{x}_{k}$ in terms of optimality:
\begin{align*}
    & \mathbb{E}[(f(\hat{x}_{k}) - f^*)] \le \frac{\|x_0 - \bar{x}_0\|^2}{2 S_k} + C_{\beta, c, \mathcal{B}_h}B_f^2 \frac{\sum_{t=0}^{k-1} \alpha_t^2}{2 S_k}.
\end{align*}
Finally, \red{by using the fact that $\hat{w}_{k} \in \mathcal{X}$, we obtain} the following bound  in expectation for the average sequence $\hat{x}_{k}$ in terms of feasibility violation of the constraints:
\begin{align*}
    & \mathbb{E}[\dist^2 (\hat{x}_{k}, \mathcal{X})] \le \mathbb{E}[\|\hat{x}_{k} - \hat{w}_{k}\|^2 ]\\
    & \le \frac{\alpha_0 c \mathcal{B}_h^2}{\beta  (1 - \beta)}\left( \frac{\|x_0 - \bar{x}_0\|^2}{S_k} + C_{\beta, c, \mathcal{B}_h} B_f^2 \frac{\sum_{t=0}^{k-1} \alpha_t^2}{S_k} \right).
\end{align*}
Hence, the statements of the theorem are proved.
\end{proof}

\noindent \red{Now,  Theorem \ref{Th:convex_case2SMBA} yields (sublinear)  convergence rates for SMBA iterates  under convex objective if the non-increasing stepsize $\alpha_k$ satisfies e.g., the conditions: $\sum_{t=0}^{\infty}\alpha_t = \infty$ and $\sum_{t=0}^{\infty}\alpha_t^2 < \infty \text{ or } \sum_{t=0}^{k-1}\alpha_t^2 \sim \mathcal{O} (\ln{(k+1)})$ for all $k \geq 1$.} Let us discuss a few possible choices for such stepsize $\alpha_k$:
\begin{enumerate}
    \item  Consider  $\alpha_k = \red{\frac{\alpha_0}{\sqrt{k+2} \ln (k+2)}}\;\;  \forall k \ge 1$, with {$\alpha_0 >0$}.  Note that this choice  yields:
    \begin{align*}
        & \sum_{t=1}^{k+1} \alpha_t \ge \red{\frac{\alpha_0 (k+1)}{\sqrt{k+3} \ln (k+3)}}  \;\; \text{and } \;\; \sum_{t=1}^{k+1} \alpha_t^2 \le \frac{\alpha_0^2}{\red{\ln(3)}}.
    \end{align*}
     Thus, from Theorem \ref{Th:convex_case2SMBA}, we obtain the following sublinear convergence rates:
    \begin{align*}
        & \mathbb{E}[(f(\hat{x}_{k}) - f^*)] \le \mathcal{O}\left( \frac{\ln (k+3)}{\sqrt{k+1}} \right),  \;\;  \mathbb{E}[\dist^2 (\hat{x}_{k}, \mathcal{X})] \le \mathcal{O}\left(\frac{\ln (k+3)}{\sqrt{k+1}} \right).
    \end{align*}
    \item Other choice  is $\alpha_k = \frac{\alpha_0}{\sqrt{k}} \;\;  \forall k \ge 1$ with {$\alpha_0 >0$}.  This choice  gives us:
    \begin{align*}
        & \sum_{t=1}^{k+1} \alpha_t \ge \alpha_0 \sqrt{k+1}  \;\; \text{and } \;\;  \sum_{t=1}^{k+1} \alpha_t^2 \le \mathcal{O} (\alpha_0^2 \ln (k+1)).
    \end{align*}
    Thus, we have the following  sublinear convergence rates from Theorem \ref{Th:convex_case2SMBA}:
    \begin{align*}
        & \mathbb{E}[(f(\hat{x}_{k}) - f^*] \le \mathcal{O}\left(\frac{1}{\sqrt{k+1}} + \frac{\ln (k+1)}{\sqrt{k+1}} \right), \\
        & \mathbb{E}[\dist^2 (\hat{x}_{k}, \mathcal{X})] \le \mathcal{O}\left(\frac{1}{\sqrt{k+1}} + \frac{\ln (k+1)}{\sqrt{k+1}} \right).
    \end{align*}
\end{enumerate} 
\noindent It is worth mentioning that although both previous choices for the stepsize $\alpha_k$ give the same theoretical rates, in our practical implementations we observed that the first choice, i.e., $\alpha_k = \frac{\alpha_0}{L_f \sqrt{k+2} \ln (k+2)}$,  performs better than the second choice $\alpha_k = \frac{\alpha_0}{\sqrt{k}}$.


\subsection{Convergence rates for SMBA under strong convex objective}
In this section, we consider that the objective function $f$ satisfies the strong convexity condition given in  \eqref{as:strg_conv_fsmooth} {(with $\mu >0$)} and obtain the convergence rates for SMBA. Note that due to the strong convexity assumption on $f$, we have a unique global minima for \eqref{eq:probsmooth}, denoted as $x^*$, i.e., in relation \eqref{eq:main_reccSMBA} we have $\bar{x}_{k+1} = \bar{v}_k = x^*$ for all $k$. First, we provide a main recurrence and use it later to prove convergence rates for the SMBA algorithm in these settings.

\begin{lemma}\label{lemma_main_recc_strconvSMBA}
Let Assumptions \ref{assumption1smooth} {(with $\mu>0$)}, \ref{assumption2smooth} and \ref{assumption3smooth} hold  and the bounds from \eqref{eq:bound_xksmooth} be valid. Also $h(\cdot, \xi)$ for all $\xi \in \Omega$ be convex functions. Choose $\beta \in (0,1)$ and  {$\alpha_k = \frac{2}{\mu (k+1)} > 0$ for all $k\ge 0$}. Then, we have the following recurrence true for the iterates generated by SMBA algorithm:
        \begin{align}
		&(k+1)^2\mathbb{E}[\|x_{k+1} - x^*\|^2] \le k^2\mathbb{E}[\|x_k - x^*\|^2] +  \frac{4}{\mu^2}  \hat{C}_{\beta, B_f, \mathcal{B}_h} \label{eq:main_recc_strconvSMBA}\\
        & \; - 2(k+1) \mathbb{E}[\|\Pi_{\mathcal{X}}(x_k) - \bar{x}_k\|^2]-  \frac{\beta  (1 - \beta) (k+1)^2}{2 c \mathcal{B}_h^2} \mathbb{E}[\dist^2 (x_{k}, \mathcal{X})] \nonumber,
	\end{align}
 where $\hat{C}_{\beta, B_f, \mathcal{B}_h} = \left( {\frac{2 c \mathcal{B}_h^2}{\beta (1 - \beta)}\bar{B}_f^2 + \left(1 + \frac{2\beta (1 - \beta)}{c \mathcal{B}_h^2}\right)  B_f^2} \right)$.
\end{lemma}

\begin{proof} 

{ From the basic properties of the projection and recalling  that $x^* \in \mathcal{X}^* \subseteq \mathcal{Y}$ and $u_k = x_k - \alpha_k \nabla f(x_k)$}, we have:
\begin{align*}
    &\|v_k - x^*\|^2 \; {\overset{\eqref{eq:alg2step1SMBA}}{=} \;  \|\Pi_\mathcal{Y} (u_k) - \Pi_\mathcal{Y} (x^*)\|^2 \le \|u_k - x^*\|^2} \nonumber\\
    & = \|x_k - x^*\|^2 - 2 \alpha_k    \langle \nabla f(x_k), x_k - x^* \rangle + \alpha_k^2\|\nabla     f(x_k)\|^2\nonumber\nonumber\\
    & \overset{\eqref{as:strg_conv_fsmooth}}{\le} (1 - \mu \alpha_k) \|x_k - x^*\|^2 - 2 \alpha_k \left( f({x_k}) - f(x^*) \right) + \alpha_k^2\|\nabla     f(x_k)\|^2 \nonumber\\
    & \le (1 - \mu \alpha_k) \|x_k - x^*\|^2 - 2 \alpha_k \left( f(\Pi_\mathcal{X}(x_k)) + \langle \nabla f(\Pi_\mathcal{X}(x_k)), {x_k} - \Pi_\mathcal{X}(x_k) \rangle - f(x^*) \right)\nonumber\\
    & \quad + \alpha_k^2\|\nabla     f(x_k)\|^2 \nonumber \\
    & \le (1 - \mu \alpha_k) \|x_k - x^*\|^2 - 2 \alpha_k \left( f(\Pi_\mathcal{X}(x_k)) - f(x^*) \right)\nonumber\\
    & \qquad + \frac{{1}}{\eta} \|\nabla f(\Pi_{\mathcal{X}} (x_k))\|^2 \alpha_k^2 + \eta \|x_k - \Pi_\mathcal{X}(x_k)\|^2 + \alpha_k^2\|\nabla     f(x_k)\|^2 \nonumber\\
    & \overset{\eqref{eq:bound_xksmooth}}{\le} (1 - \mu \alpha_k) \|x_k - x^*\|^2 - 2 \alpha_k \left( f(\Pi_\mathcal{X}(x_k)) - f(x^*) \right) + \eta \|x_k - \Pi_\mathcal{X}(x_k)\|^2 \\
    & \quad \qquad+ \left(\frac{{1}}{\eta} \bar{B}_f^2 +  B_f^2\right) \alpha_k^2,\nonumber
\end{align*}

\noindent where third inequality uses the convexity of $f$ and the fourth inequality follows from the identity $2\langle a, b \rangle \le \frac{1}{\eta} \|a\|^2 + \eta \|b\|^2$, for all $a,b \in \mathbb{R}^n$ and $ \eta > 0$. Now, choosing $\eta = \frac{\beta (1 - \beta)}{2 c \mathcal{B}_h^2}$ and using it in \eqref{eq:main_reccSMBA}, we get:
\begin{align*}
     & \mathbb{E}_{\xi_k}[\|x_{k+1} - \bar{x}_{k+1}\|^2] \le (1 - \mu \alpha_k) \|x_k - x^*\|^2 - 2 \alpha_k \left( f(\Pi_\mathcal{X}(x_k))- f(x^*) \right)  \\
     & \quad - \frac{\beta (1 - \beta)}{2 c \mathcal{B}_h^2} \dist^2 (x_{k}, \mathcal{X}) + \left( {\frac{2 c \mathcal{B}_h^2}{\beta (1 - \beta)}\bar{B}_f^2 + \left(1 + \frac{2\beta (1 - \beta)}{c \mathcal{B}_h^2}\right)  B_f^2} \right)\alpha_k^2.
\end{align*}
Taking full expectation and denoting $\hat{C}_{\beta, B_f, \mathcal{B}_h} = \left( {\frac{2 c \mathcal{B}_h^2}{\beta (1 - \beta)}\bar{B}_f^2 + \left(1 + \frac{2\beta (1 - \beta)}{c \mathcal{B}_h^2}\right)  B_f^2} \right)$, we obtain:
\begin{align} \label{eq:waySMBA}
    & \mathbb{E}[\|x_{k+1} - x^*\|^2]  \le  (1 - \mu \alpha_k)\mathbb{E}[\|x_k - x^*\|^2] - 2 \alpha_k \mathbb{E}[(f(\Pi_\mathcal{X}(x_k)) - f(x^*))] \nonumber\\
    & \qquad  -   \frac{\beta (1 - \beta)}{2 c \mathcal{B}_h^2} \mathbb{E} [\dist^2 (x_{k}, \mathcal{X})]+ \hat{C}_{\beta, B_f, \mathcal{B}_h}\alpha^2_k \nonumber\\
    & \overset{\eqref{eq:QGsmooth}}{\le} (1 - \mu \alpha_k)\mathbb{E}[\|x_k - x^*\|^2] - \mu \alpha_k \mathbb{E}[\|\Pi_{\mathcal{X}}(x_k) - x^*\|^2] \\
    & \qquad -  \frac{\beta (1 - \beta)}{2 c \mathcal{B}_h^2} \mathbb{E} [\dist^2 (x_{k}, \mathcal{X})] + \hat{C}_{\beta, B_f, \mathcal{B}_h}\alpha^2_k.\nonumber
\end{align} 
From the choice of $\alpha_k$, i.e., $\alpha_k = \frac{2}{\mu(k+1)}$,  \eqref{eq:waySMBA} yields:\vspace{-0.3cm}
\begin{align*}
    \mathbb{E}[\|x_{k+1} - x^*\|^2]  & \le \left(1 -\frac{2}{k+1}\right)\mathbb{E}[\|x_k - x^*\|^2] - \frac{2}{(k+1)} \mathbb{E}[\|\Pi_{\mathcal{X}}(x_k) - x^*\|^2] \nonumber\\[-3pt]
    & \qquad  -  \frac{\beta  (1 - \beta)}{2 c \mathcal{B}_h^2} \mathbb{E}[\dist^2 (x_{k}, \mathcal{X})] +  \frac{4}{\mu^2 (k+1)^2} \hat{C}_{\beta, B_f, \mathcal{B}_h}.
\end{align*}
After multiplying the whole inequality by $(k+1)^2$, and using  $k^2 - 1\le k^2$, we obtain:
\vspace{-0.3cm}
\begin{align*}
    & (k+1)^2\mathbb{E}[\|x_{k+1} - x^*\|^2] \le k^2 \mathbb{E}[\|x_k - x^*\|^2] + \frac{4}{\mu^2} \hat{C}_{\beta, B_f, \mathcal{B}_h}\\
    & \quad - 2(k+1) \mathbb{E}[\|\Pi_{\mathcal{X}}(x_k) - x^*\|^2]-  \frac{\beta  (1 - \beta) (k+1)^2}{2 c \mathcal{B}_h^2} \mathbb{E}[\dist^2 (x_{k}, \mathcal{X})].
\end{align*}
Hence, we get the desired result \eqref{eq:main_recc_strconvSMBA}.
\end{proof}

\noindent Before deriving the convergence rates, let us define the sum: \vspace{-0.3cm}
\begin{align*}
	{\bar{S}_k = \sum_{t=0}^{k-1} (t+1)^2 \sim \mathcal{O} (k^3 ) \quad \forall k \geq 0,}
\end{align*}
and  the corresponding average sequences: \vspace{-0.3cm}
\begin{align*}
	&{\hat{x}_k = \frac{\sum_{t=0}^{k-1} (t+1)^2 x_{t}}{\bar{S}_k},\;\;	
	\hat{w}_k = \frac{1}{\bar{S}_k} \sum_{t=0}^{k-1} (t+1)^2 \Pi_{\mathcal X}(x_{t}) \in \mathcal{X}.} 
\end{align*}

\noindent The following theorem provides the convergence rates for the SMBA algorithm when the objective is strongly convex.\\

\begin{theorem}\label{Th:strconvex_caseSMBA}
   Let Assumptions \ref{assumption1smooth} (with $\mu>0$), \ref{assumption2smooth} and \ref{assumption3smooth} hold  and the bounds from \eqref{eq:bound_xksmooth} be valid. Further, choose  stepsizes $\beta \in (0, 1)$ and {$\alpha_k = \frac{2}{\mu (k+1)} > 0$, for all $k\ge 0$}. Then, we have the following (sublinear) convergence rates for the average sequence $\hat{x}_{k}$ in terms of optimality and feasibility violation for any $k \ge 0$:
\begin{align*}
    & \mathbb{E}[\|\hat{x}_{k} - x^*\|^2 ] \le \mathcal{O}\left( \frac{\hat{C}_{\beta,B_f,\mathcal{B}_h}}{\mu^2 (k+1)} + \frac{c \mathcal{B}_h^2 \hat{C}_{\beta,B_f,\mathcal{B}_h}}{\beta  (1 - \beta) \mu^2 (k+1)^2} \right),\\
    & \mathbb{E}[\dist^2 (\hat{x}_{k}, \mathcal{X})] \le \mathcal{O} \left(\frac{c \mathcal{B}_h^2 \hat{C}_{\beta, B_f,\mathcal{B}_h}}{\beta  (1 - \beta) \mu^2 (k+1)^2 } \right).
\end{align*}
\end{theorem}

\begin{proof}
For any {$k \ge 0$}, we have the relation \eqref{eq:main_recc_strconvSMBA} from Lemma \ref{lemma_main_recc_strconvSMBA}. Summing it from $0$ to {$k-1$}, we obtain:
\begin{align*} 
    k^2 \mathbb{E}[\|x_{k} - x^*\|^2] & \le - 2 \sum_{t= {0}}^{k-1} (t+1) \mathbb{E}[\|\Pi_{\mathcal{X}} (x_t) - x^*\|^2] \\
    & \quad  -  \frac{\beta  (1 - \beta)}{2c \mathcal{B}_h^2} \sum_{t={0}}^{k-1} (t+1)^2 \mathbb{E}[\|x_t - \Pi_\mathcal{X}(x_t)\|^2] \!+\! \frac{4}{\mu^2} \hat{C}_{\beta,B_f,\mathcal{B}_h} {k}.
\end{align*}
\red{Now, using the fact that $\frac{(t+1)^2}{(k+1)}  \le  (t+1)$ for all {$t = 0 : k$}, we have:}
\vspace{-0.3cm}
\begin{align*} 
    \red{k^2 \mathbb{E}[\|x_{k} - x^*\|^2]} & \red{\le  - 2 \frac{1}{(k+1)} \sum_{t= {0}}^{k-1} (t+1)^2 \mathbb{E}[\|\Pi_{\mathcal{X}} (x_t) - x^*\|^2]} \\
    & \red{\quad  -  \frac{\beta  (1 - \beta)}{2c \mathcal{B}_h^2} \sum_{t= {0}}^{k-1} (t+1)^2 \mathbb{E}[\|x_t - \Pi_\mathcal{X}(x_t)\|^2] + \frac{4}{\mu^2} \hat{C}_{\beta,B_f,\mathcal{B}_h} {k}}\\
    & \red{= - 2 \frac{\bar{S}_k}{\bar{S}_k (k+1)} \sum_{t= {0}}^{k-1} (t+1)^2 \mathbb{E}[\|\Pi_{\mathcal{X}} (x_t) - x^*\|^2] }\\
    & \red{\quad  -  \frac{\bar{S}_k \beta  (1 - \beta)}{2 \bar{S}_k c \mathcal{B}_h^2} \sum_{t= {0}}^{k-1} (t+1)^2 \mathbb{E}[\|x_t - \Pi_\mathcal{X}(x_t)\|^2] + \frac{4}{\mu^2} \hat{C}_{\beta,B_f,\mathcal{B}_h} {k}.}
\end{align*}
\red{Now, using non-negativity of $(k+1)^2 \mathbb{E}[\|x_{k+1} - x^*\|^2$, the linearity of the expectation operator, convexity of $\|\cdot\|^2$ and the expressions of $\hat{w}_{k}$ and $ \hat{x}_{k}$,} we further get: 
\begin{align*} 
	& \frac{2\bar{S}_k}{ (k+1)} \mathbb{E}[\|\hat{w}_{k} - x^* \|^2] + \frac{\beta  (1 - \beta) \bar{S}_k}{2 c \mathcal{B}_h^2}\mathbb{E}[\|\hat{x}_{k} - \hat{w}_{k}\|^2] \le \frac{4}{\mu^2} \hat{C}_{\beta,B_f,\mathcal{B}_h} {k}.
\end{align*}
After some simple calculations and keeping only the dominant terms, we have:
\begin{align*}
    & \mathbb{E}[\|\hat{w}_{k} - x^* \|^2] \le \mathcal{O} \left( \frac{\hat{C}_{\beta,B_f,\mathcal{B}_h}}{\mu^2 (k + 1)} \right),\\
    & \mathbb{E}[\|\hat{x}_{k} - \hat{w}_{k}\|^2 ]\le \mathcal{O} \left( \frac{c \mathcal{B}_h^2 \hat{C}_{\beta,B_f,\mathcal{B}_h}}{\beta  (1 - \beta) \mu^2 (k+1)^2 } \right).
\end{align*}
Since $\hat{w}_k \in \mathcal{X}$, we get the following convergence rate for the average sequence $\hat{x}_{k}$ in terms of feasibility violation:
\begin{align*}
    \mathbb{E}[\dist^2 (\hat{x}_{k}, \mathcal{X})] \le \mathbb{E}[\|\hat{x}_{k} - \hat{w}_{k}\|^2 ]\le \mathcal{O} \left( \frac{c \mathcal{B}_h^2 \hat{C}_{\beta,B_f, \mathcal{B}_h}}{\beta  (1 - \beta) \mu^2 (k+1)^2 } \right).
\end{align*}
Moreover, we get the convergence rate for average sequence $\hat{x}_k$ in terms of optimality:
\begin{align*}
    \mathbb{E}[\|\hat{x}_{k} - x^*\|^2 ] & \le 2 \mathbb{E}[\|\hat{x}_{k} - \hat{w}_{k}\|^2 ] + 2 \mathbb{E}[\|\hat{w}_{k} - x^*\|^2 ]\\
    & \le \mathcal{O}\left( \frac{\hat{C}_{\beta,B_f,\mathcal{B}_h}}{\mu^2 (k + 1)} + \frac{c \mathcal{B}_h^2 \hat{C}_{\beta,B_f,\mathcal{B}_h}}{\beta  (1 - \beta) \mu^2 (k+1)^2} \right).
\end{align*}
Hence, we obtain the claimed results.
\end{proof}

\noindent  The convergence rates from Theorem \ref{Th:convex_case2SMBA}  (Theorem \ref{Th:strconvex_caseSMBA})  are new
and valid for the convex settings,   i.e., when the objective function and the functional constraints are both convex in problem \eqref{eq:probsmooth}. In these settings,  we obtain convergence rates $\mathcal{O} (k^{-1/2})$ or $\mathcal{O} (k^{-1})$  in terms of optimality and feasibility for an average sequence, for convex or strongly convex objective, respectively. Also, it is important to note that,  according to our convergence analysis,  we need to choose the stepsize $\beta \in (0,1)$, although in numerical simulations (see Section 5 below) we observed that algorithm SMBA works also for $\beta \in [1,2)$. It is an open question how to modify our current convergence analysis to cover a larger interval choice for $\beta$. \textit{To the best of our knowledge, this paper is the first providing a simple stochastic  primal gradient algorithm for solving convex smooth problems of the form \eqref{eq:probsmooth} that has mathematical guarantees of convergence of order $\mathcal{O} (1/\sqrt{k})$/$\mathcal{O} (1/k)$, \red{thus matching the convergence rates obtained in some other papers from the literature on stochastic (primal-dual) (sub)gradient methods, see e.g.,  \cite{NecSin:22, NecNed:21, Ned:11, WanBer:16}.}}



\section{Numerical results}
In this section, we test the performance of our algorithm on convex quadratically constrained quadratic programs (QCQPs) given in  \eqref{qcqp} using synthetic and real data. 
We compare our Stochastic Moving Ball Approximation (SMBA) algorithm to its deterministic counterparts, MBA \red{and MBA-AS} algorithms \cite{AusTeb:10}, and two dedicated commercial optimization software packages, Gurobi \cite{Gurobi} and FICO \cite{FICO} (both have specialized solvers for QCQPs). The implementation details are conducted using MATLAB R2023b on a laptop equipped with an i5 CPU operating at 2.1 GHz and 16 GB of RAM.

\vspace{-0.3cm}
\subsection{Solving standard convex QCQPs} 
\noindent To solve QCQPs \eqref{qcqp}, we consider two settings for the objective function, i.e., convex and strongly convex, and convex functional constraints. Also, we consider $\mathcal{Y} = \mathbb{R}^n_+$, and to generate the synthetic data we follow the same technique as in \cite{AusTeb:10}. This means to generate $m$ convex constraints in \eqref{qcqp}, choosing $Q_i = Y_i^T D_i Y_i$, where $Y_i$'s are randomly generated orthogonal $n\times n$ matrices using \texttt{orth} in Matlab and each $D_i$ is a diagonal matrix with $n/10$ zero entries on the diagonal and the rest of the entries are randomly generated in $(0, 1)$. Moreover, for the convex objective, the positive semidefinite matrix $Q_f = Y_f^T D_f Y_f$ is generated in the same fashion as  $ Q_i$'s, and for the strongly convex case, we generate the positive definite matrix $Q_f = Y_f^T D_f Y_f$ with each entry in the diagonal of the matrix $D_f$ to be nonzero ranging in $(0,1)$. Furthermore, the vectors $q_f, q_i \in \mathbb{R}^n$, for all $i = 1:m$, are generated from uniform distributions. \red{For the scalars $b_i\in \mathbb{R}$ for all $i = 1:m$, we have considered two scenarios: for the first half of the table we generated uniformly at random an initial point $x_0$ and then chose $b_i$'s as $b_i = \frac{1}{2} x_0^T Q_i x_0 + q_i^T x + 0.1$; for the second half of the table we generated $b_i$'s uniformly at random. Note that in both scenarios the QCQPs are feasible, since in the first case $x_0$ is a feasible point, while in the second case $0$ is always  a feasible point. }

\medskip

\noindent In order to solve QCQPs \eqref{qcqp} using SMBA, we set the stepsizes {$\alpha_k = \frac{2}{\mu (k+1)}$} and  $\beta = 0.96, \text{ or }1.96$ for the strongly convex objective;  $\alpha_k = \frac{1}{L_f\sqrt{k+2}\ln{(k+2)}}$ and  $\beta = 0.96, \text{ or }1.96$ for the convex objective. Moreover, we choose $p_{v_k, \xi_k}$ as defined in \eqref{eq:alg2step2SMBA}. We stop SMBA when  $\|\max(0, h(x))\|^2 \le 10^{-2}$ and $|f(x) - f^*| \leq 10^{-2}$ (with $f^*$ computed via FICO solver) or when $\max (\|x_{k+1} - x_k\|^2, \ldots, \|x_{k-M+1} - x_{k-M}\|^2) \le 10^{-3}$, with  $M=10$ (when FICO does not solve the QCQP due to license limitations). Furthermore, MBA \red{and MBA-AS} being feasible at each iteration, we stop them when $|f(x) - f^*| \leq 10^{-2}$ (with $f^*$ computed via FICO) or $\max (\|x_{k+1} - x_k\|^2, \ldots, \|x_{k-M+1} - x_{k-M}\|^2) \le 10^{-3}$, with  $M=10$ (when FICO does not solve the problem). 



\begin{landscape}
    \begin{table}[ht]
\centering
\caption{\noindent Comparison of SMBA (two different choices of $\beta$), MBA \red{and MBA-AS} (when $x_0$ is feasible), Gurobi, and FICO solvers in terms of CPU time (in sec.) on QCQPs having various dimensions $n, m$ and with feasible initial point (strongly convex and convex objective) and with infeasible initial point (strongly convex and convex objective).}
\label{table}
\begin{tabular}{|ccccccccc|}
\hline
\multicolumn{1}{|c|}{\multirow{2}{*}{\textbf{\begin{tabular}[c]{@{}c@{}}Data\\ (n, m)\end{tabular}}}} & \multicolumn{4}{c|}{SMBA}                                                                                                                                                              & \multicolumn{1}{c|}{\multirow{2}{*}{MBA}} & \multicolumn{1}{c|}{\multirow{2}{*}{\red{MBA-AS}}} & \multicolumn{1}{c|}{\multirow{2}{*}{Gurobi}} & \multicolumn{1}{l|}{\multirow{2}{*}{FICO}} \\ 
\multicolumn{1}{|c|}{}                                                                                & \multicolumn{1}{c|}{$\beta = 0.96$}       & \multicolumn{1}{c|}{std}                  & \multicolumn{1}{c|}{$\beta = 1.96$}                & \multicolumn{1}{c|}{std}                  & \multicolumn{1}{c|}{}      & \multicolumn{1}{c|}{}               & \multicolumn{1}{c|}{}                        & \multicolumn{1}{l|}{}                      \\ \hline
\multicolumn{9}{|c|}{Initial point $x_0$ is feasible}                                                                                                                                                \\ \hline
\multicolumn{1}{|c|}{$10^2, 10^2$}                                                                    & \multicolumn{1}{c|}{$2.2 \times 10^{-3}$} & \multicolumn{1}{l|}{$1.8 \times 10^{-3}$} & \multicolumn{1}{l|}{$\mathbf{1.7 \times 10^{-3}}$} & \multicolumn{1}{l|}{$0.7 \times 10^{-3}$} & \multicolumn{1}{c|}{0.08}   & \multicolumn{1}{c|}{\red{0.01}}              & \multicolumn{1}{c|}{2.22}                    & 1.02                                       \\ \hline
\multicolumn{1}{|c|}{$10^2, 10^3$}                                                                    & \multicolumn{1}{c|}{0.02}                 & \multicolumn{1}{c|}{$1.6 \times 10^{-3}$} & \multicolumn{1}{c|}{\textbf{0.01}}                 & \multicolumn{1}{c|}{$0.6 \times 10^{-3}$} & \multicolumn{1}{c|}{2.20}     & \multicolumn{1}{c|}{\red{0.24}}           & \multicolumn{1}{c|}{27.42}                   & 2.61                                       \\ \hline
\multicolumn{1}{|l|}{$10^2, 5\times 10^3$}                                                            & \multicolumn{1}{c|}{\textbf{0.04}}        & \multicolumn{1}{c|}{$2.5 \times 10^{-3}$} & \multicolumn{1}{c|}{0.05}                          & \multicolumn{1}{c|}{$7.5 \times 10^{-3}$} & \multicolumn{1}{c|}{134.59}     & \multicolumn{1}{c|}{\red{2.24}}          & \multicolumn{1}{c|}{119.90}                  & 80.83                                      \\ \hline
\multicolumn{1}{|c|}{$10^3, 10^2$}                                                                    & \multicolumn{1}{c|}{\textbf{0.08}}        & \multicolumn{1}{c|}{$5.8 \times 10^{-3}$} & \multicolumn{1}{c|}{0.12}                          & \multicolumn{1}{c|}{0.1}                  & \multicolumn{1}{c|}{3.5}     & \multicolumn{1}{c|}{\red{1.71}}             & \multicolumn{1}{c|}{$1.7\times 10^3$}        & 11.47                                      \\ \hline
\multicolumn{1}{|c|}{$10^3, 10^3$}                                                                    & \multicolumn{1}{c|}{\textbf{0.56}}        & \multicolumn{1}{c|}{0.06}                 & \multicolumn{1}{c|}{0.67}                          & \multicolumn{1}{c|}{0.04}                 & \multicolumn{1}{c|}{73.21}       & \multicolumn{1}{c|}{\red{21.77}}        & \multicolumn{1}{c|}{$2.2\times 10^4$}        & 234.21                                     \\ \hline
\multicolumn{1}{|c|}{$10^3, 5\times10^3$}                                                             & \multicolumn{1}{c|}{2.96}                 & \multicolumn{1}{c|}{0.03}                 & \multicolumn{1}{c|}{\textbf{2.76}}                 & \multicolumn{1}{c|}{0.13}                 & \multicolumn{1}{c|}{830.65}    & \multicolumn{1}{c|}{\red{92.47}}           & \multicolumn{1}{c|}{*}                       & 606.96                                     \\ \hline
\multicolumn{1}{|c|}{$10^3, 10^4$}                                                                    & \multicolumn{1}{c|}{\textbf{9.41}}        & \multicolumn{1}{c|}{2.69}                 & \multicolumn{1}{c|}{9.79}                          & \multicolumn{1}{c|}{3.96}                 & \multicolumn{1}{c|}{$1.07\times 10^4$}  & \multicolumn{1}{c|}{\red{$1.76\times 10^3$}}  & \multicolumn{1}{c|}{*}                       & *                                          \\ \hline
\hline


\multicolumn{1}{|c|}{$10^2, 10^2$}                                                                    & \multicolumn{1}{c|}{\textbf{0.03}}        & \multicolumn{1}{c|}{$6.1 \times 10^{-3}$} & \multicolumn{1}{c|}{0.04}                          & \multicolumn{1}{c|}{$6.4 \times 10^{-3}$} & \multicolumn{1}{c|}{0.29}      & \multicolumn{1}{c|}{\red{0.04}}           & \multicolumn{1}{c|}{6.26}                    & 1.28                                       \\ \hline
\multicolumn{1}{|c|}{$10^2, 10^3$}                                                                    & \multicolumn{1}{c|}{\textbf{0.01}}        & \multicolumn{1}{c|}{$5 \times 10^{-3}$}   & \multicolumn{1}{c|}{0.02}                          & \multicolumn{1}{c|}{$3.3 \times 10^{-3}$} & \multicolumn{1}{c|}{3.97}     & \multicolumn{1}{c|}{\red{0.47}}            & \multicolumn{1}{c|}{10.86}                   & 3.91                                       \\ \hline
\multicolumn{1}{|c|}{$10^2, 5 \times 10^3$}                                                           & \multicolumn{1}{c|}{\textbf{0.06}}        & \multicolumn{1}{c|}{$9.4 \times 10^{-3}$} & \multicolumn{1}{c|}{0.18}                          & \multicolumn{1}{c|}{0.28}                 & \multicolumn{1}{c|}{281.17}     & \multicolumn{1}{c|}{\red{3.25}}          & \multicolumn{1}{c|}{175.21}                   & 9.03                                       \\ \hline
\multicolumn{1}{|c|}{$10^3, 10^2$}                                                                    & \multicolumn{1}{c|}{\textbf{2.28}}        & \multicolumn{1}{c|}{0.04}                 & \multicolumn{1}{c|}{2.41}                          & \multicolumn{1}{c|}{0.14}                 & \multicolumn{1}{c|}{7.93}         & \multicolumn{1}{c|}{\red{3.05}}        & \multicolumn{1}{c|}{$1.14\times 10^3$}       & 39.86                                      \\ \hline
\multicolumn{1}{|c|}{$10^3, 10^3$}                                                                    & \multicolumn{1}{c|}{3.18}                 & \multicolumn{1}{c|}{0.95}                 & \multicolumn{1}{c|}{\textbf{2.03}}                 & \multicolumn{1}{c|}{0.03}                 & \multicolumn{1}{c|}{263.62}      & \multicolumn{1}{c|}{\red{29.49}}         & \multicolumn{1}{c|}{*}                       & 401.75                                     \\ \hline
\multicolumn{1}{|c|}{$10^3, 5 \times 10^3$}                                                           & \multicolumn{1}{c|}{76.83}                & \multicolumn{1}{c|}{3.25}                 & \multicolumn{1}{c|}{\textbf{72.04}}                & \multicolumn{1}{c|}{1.01}                 & \multicolumn{1}{c|}{$1.9\times 10^3$}   & \multicolumn{1}{c|}{\red{142.58}}  & \multicolumn{1}{c|}{*}                       & $1.07\times 10^3$                          \\ \hline
\multicolumn{1}{|c|}{$10^3, 10^4$}                                                                    & \multicolumn{1}{c|}{332.04}               & \multicolumn{1}{c|}{7.94}                 & \multicolumn{1}{c|}{\textbf{289.58}}               & \multicolumn{1}{c|}{5.02}                 & \multicolumn{1}{c|}{$2.09\times 10^4$}  & \multicolumn{1}{c|}{\red{$4.12\times 10^3$}}   & \multicolumn{1}{c|}{*}                       & *                                          \\ \hline
\multicolumn{9}{|c|}{Initial point $x_0$ is infeasible}                                                                                                                                     
   \\ \hline
\multicolumn{1}{|c|}{$10^2, 10^2$}                                                                    & \multicolumn{1}{c|}{0.97}                 & \multicolumn{1}{c|}{0.50}                 & \multicolumn{1}{c|}{\textbf{0.88}}                 & \multicolumn{1}{c|}{0.24}                 & \multicolumn{1}{c|}{*}      & \multicolumn{1}{c|}{*}             & \multicolumn{1}{c|}{5.94}                    & 1.24                                       \\ \hline
\multicolumn{1}{|c|}{$10^2, 10^3$}                                                                    & \multicolumn{1}{c|}{21.76}                & \multicolumn{1}{c|}{3.20}                 & \multicolumn{1}{c|}{8.56}                          & \multicolumn{1}{c|}{4.02}                 & \multicolumn{1}{c|}{*}        & \multicolumn{1}{c|}{*}            & \multicolumn{1}{c|}{64.05}                   & \textbf{4.36}                              \\ \hline
\multicolumn{1}{|c|}{$10^2, 5 \times 10^3$}                                                           & \multicolumn{1}{c|}{18.50}                & \multicolumn{1}{c|}{3.01}                 & \multicolumn{1}{c|}{\textbf{3.35}}                 & \multicolumn{1}{c|}{1.64}                 & \multicolumn{1}{c|}{*}      & \multicolumn{1}{c|}{*}              & \multicolumn{1}{c|}{326.44}                  & 12.66                                      \\ \hline
\multicolumn{1}{|c|}{$10^3, 10^2$}                                                                    & \multicolumn{1}{c|}{528.49}               & \multicolumn{1}{c|}{8.92}                 & \multicolumn{1}{c|}{243.01}                        & \multicolumn{1}{c|}{7.72}                 & \multicolumn{1}{c|}{*}     & \multicolumn{1}{c|}{*}               & \multicolumn{1}{c|}{$4.59\times 10^3$}       & \textbf{22.64}                             \\ \hline
\multicolumn{1}{|c|}{$10^3, 10^3$}                                                                    & \multicolumn{1}{c|}{138.14}               & \multicolumn{1}{c|}{2.65}                 & \multicolumn{1}{c|}{\textbf{107.01}}               & \multicolumn{1}{c|}{4.1}                  & \multicolumn{1}{c|}{*}       & \multicolumn{1}{c|}{*}             & \multicolumn{1}{c|}{*}                       & 255.08                                     \\ \hline
\multicolumn{1}{|c|}{$10^3, 5\times 10^3$}                                                            & \multicolumn{1}{c|}{913.16}                     & \multicolumn{1}{c|}{7.86}                     & \multicolumn{1}{c|}{\textbf{634.59}}                              & \multicolumn{1}{c|}{9.48}                     & \multicolumn{1}{c|}{*}       & \multicolumn{1}{c|}{*}             & \multicolumn{1}{c|}{*}                       & \multicolumn{1}{c|}{759.13}                      \\ \hline
\multicolumn{1}{|c|}{$10^3, 10^4$}                                                                    & \multicolumn{1}{c|}{$2.98\times 10^3$}                     & \multicolumn{1}{c|}{6.91}                     & \multicolumn{1}{c|}{$\mathbf{1.36\times 10^3}$}                              & \multicolumn{1}{c|}{8.57}                     & \multicolumn{1}{c|}{*}      & \multicolumn{1}{c|}{*}              & \multicolumn{1}{c|}{*}                       & *                                          \\ \hline
\hline 


\multicolumn{1}{|c|}{$10^2, 10^2$}                                                                    & \multicolumn{1}{c|}{0.29}                 & \multicolumn{1}{c|}{0.08}                 & \multicolumn{1}{c|}{\textbf{0.12}}                 & \multicolumn{1}{c|}{0.03}                 & \multicolumn{1}{c|}{*}         & \multicolumn{1}{c|}{*}           & \multicolumn{1}{c|}{4.52}                    & 0.91                                       \\ \hline
\multicolumn{1}{|c|}{$10^2, 10^3$}                                                                    & \multicolumn{1}{c|}{2.18}                 & \multicolumn{1}{c|}{1.2}                  & \multicolumn{1}{c|}{\textbf{0.74}}                 & \multicolumn{1}{c|}{0.83}                 & \multicolumn{1}{c|}{*}     & \multicolumn{1}{c|}{*}               & \multicolumn{1}{c|}{27.6}                    & 2.76                                       \\ \hline
\multicolumn{1}{|c|}{$10^2, 5 \times 10^3$}                                                           & \multicolumn{1}{c|}{30.49}                & \multicolumn{1}{c|}{6.68}                 & \multicolumn{1}{c|}{\textbf{1.86}}                 & \multicolumn{1}{c|}{1.66}                 & \multicolumn{1}{c|}{*}     & \multicolumn{1}{c|}{*}               & \multicolumn{1}{c|}{251.35}                   & 9.96                                       \\ \hline
\multicolumn{1}{|c|}{$10^3, 10^2$}                                                                    & \multicolumn{1}{c|}{15.95}                & \multicolumn{1}{c|}{5.36}                 & \multicolumn{1}{c|}{\textbf{15.25}}                & \multicolumn{1}{c|}{4.74}                 & \multicolumn{1}{c|}{*}       & \multicolumn{1}{c|}{*}             & \multicolumn{1}{c|}{$7.21\times 10^3$}                   & 23.55                                      \\ \hline
\multicolumn{1}{|c|}{$10^3, 10^3$}                                                                    & \multicolumn{1}{c|}{115.94}               & \multicolumn{1}{c|}{6.51}                 & \multicolumn{1}{c|}{\textbf{104.22}}               & \multicolumn{1}{c|}{5.83}                 & \multicolumn{1}{c|}{*}     & \multicolumn{1}{c|}{*}               & \multicolumn{1}{c|}{*}                       & 230.76                                     \\ \hline
\multicolumn{1}{|c|}{$10^3, 5\times 10^3$}                                                            & \multicolumn{1}{c|}{$1.03\times 10^3$}                     & \multicolumn{1}{c|}{10.62}                     & \multicolumn{1}{c|}{698.14}                              & \multicolumn{1}{c|}{9.64}                     & \multicolumn{1}{c|}{*}         & \multicolumn{1}{c|}{*}           & \multicolumn{1}{c|}{*}                       & \multicolumn{1}{c|}{\textbf{615.76}}                      \\ \hline
\multicolumn{1}{|c|}{$10^3, 10^4$}                                                                    & \multicolumn{1}{c|}{$7.39\times 10^3$}                     & \multicolumn{1}{c|}{11.07}                     & \multicolumn{1}{c|}{$\mathbf{4.96\times 10^3}$}                              & \multicolumn{1}{c|}{9.46}                     & \multicolumn{1}{c|}{*}     & \multicolumn{1}{c|}{*}               & \multicolumn{1}{c|}{*}                       & *                                          \\ \hline
\end{tabular}
\tiny{\textbf{Note:} The $*$ in the MBA \red{and MBA-AS} column means that  MBA \red{and MBA-AS} is not applicable (the initial point is not feasible), in the Gurobi column means that Gurobi did not finish in 9 hours, and in the FICO column means that FICO was not solving the problem due to academic license limitation.}
\end{table}
\end{landscape}


\noindent In Table \ref{table} we provide the average CPU times (out of $10$ runs) in seconds together with the standard deviation (std) for the proposed algorithm SMBA, and CPU times in seconds for the state-of-the-art methods MBA \red{ and MBA-AS}, and the commercial solvers Gurobi and FICO for various dimensions $n, m$ in the QCQPs \eqref{qcqp}. The table is divided into two parts, the first part of the table presents the results when the initial point $x_0$ is chosen as feasible, while the second part shows the results when the initial point is infeasible. In both parts, the first half is for strongly convex objective and the other half is for convex objective function. One needs to observe here that SMBA outperforms drastically MBA (on large problems is even 50x faster), \red{MBA-AS (on large problems is about 10x faster),} Gurobi and FICO  when the initial point is feasible. Moreover,  for SMBA, in this case, it is tough to decide which choice of  $\beta$ works better, i.e., $\beta = 0.96 \text{ or } \beta= 1.96$.  On the other hand, if we start with an infeasible initial point, MBA \red{and MBA-AS} cannot be considered anymore for  solving the QCQPs, as they require the initial point to be feasible. Also in this scenario, SMBA always outperforms Gurobi, however FICO being not dependent on the initial point has sometimes a better behaviour than SMBA. Moreover, one can see that in this case, the better choice of $\beta$ is  usually $1.96$.

\medskip 

\begin{figure}[H]
    \centering
    \includegraphics[height=5.5cm, width=6cm]{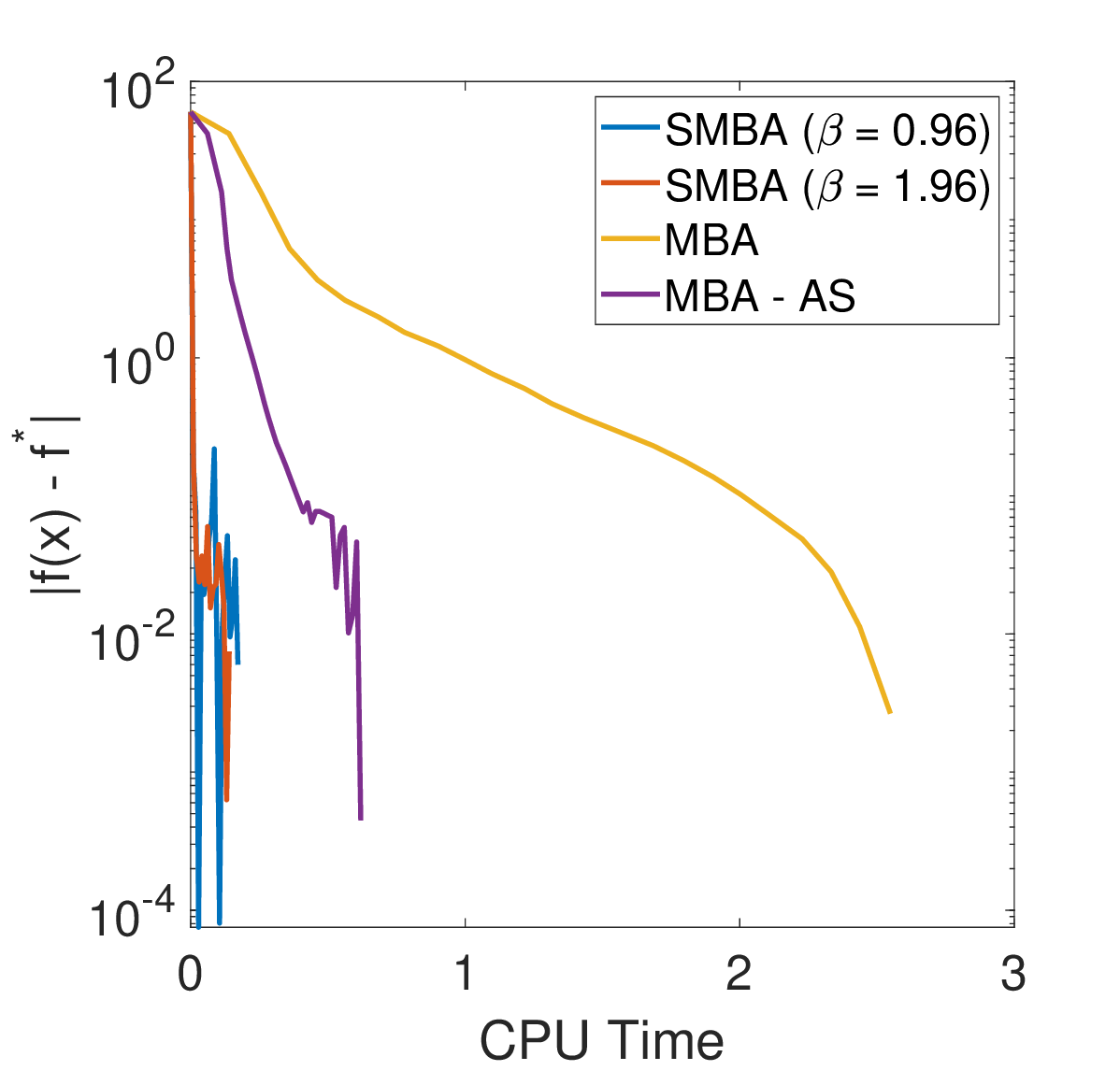}
    \includegraphics[height=5.5cm, width=6cm]{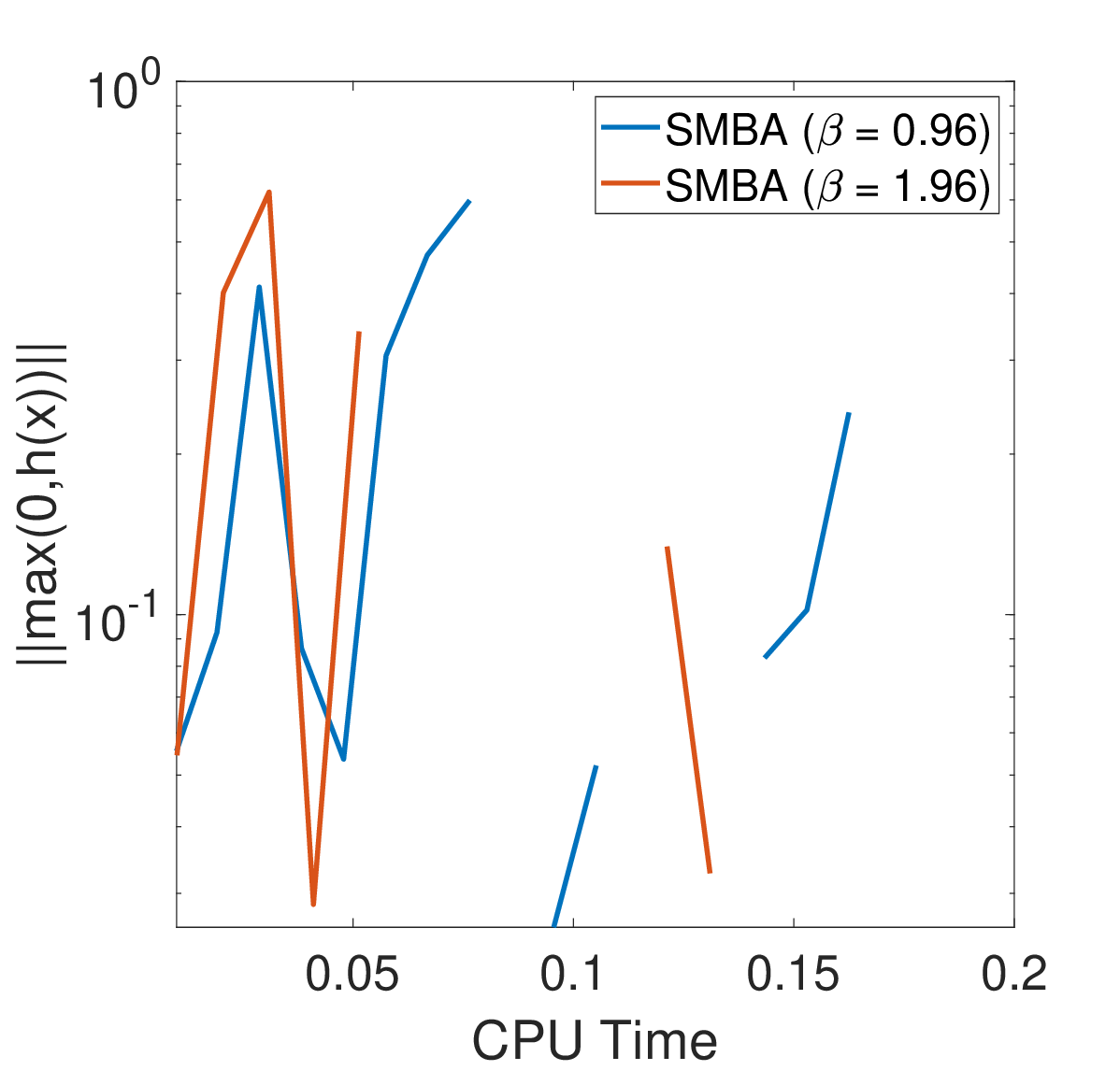}
    \caption{Behaviour of SMBA  for two choices of $\beta$ ($0.96$ and $1.96$) and comparison with MBA and MBA-AS in terms of optimality (left) and feasibility (right) for $n=100,  m = 2000$ along time (sec).}
    \label{fig:1}
\end{figure}
\noindent \red{In Figure \ref{fig:1} we plot the behaviour of SMBA algorithm (with $\beta = 0.96$ and $\beta = 1.96$) and  MBA and MBA-AS methods along time (in seconds), focusing on optimality (left) and feasibility trends (right) for dimension $n=100$ and $m=2000$ of the QCQP having  a convex objective. Notably, both MBA and MBA-AS are feasible at each iteration, so we do not plot them in the right figure. It is evident that  SMBA algorithm significantly outperforms MBA and MBA-AS, demonstrating a clear advantage in terms of optimality along time.} { Furthermore, the breaks observed in the feasibility plots occur because the SMBA algorithm is feasible at certain iterations, meaning $\|\max(0, h(x))\| = 0$. However, since the algorithm has not yet satisfied the optimality criterion at those points, it continues. During this process, it is possible for SMBA to become infeasible again, as the algorithm does not guarantee feasibility at each step.}


\subsection{Solving multiple kernel learning in support vector machine}
\noindent In this section, we test SMBA  on Support Vector Machine (SVM) with multiple kernel learning using real data, which can be formulated as a convex QCQP. Let us briefly describe the problem (our exposition follows \cite{CheLi:21}). Given a set of $N$ data points $\mathcal{S} = \{(a_{j},y_{j})\}_{j=1}^{N}$, where $a_j \in \mathbb{R}^{n_d}$ is the $n_d$-dimension input vector and $y_j \in \{-1, 1\}$ is its class label, SVM searches for a hyperplane that can best separate the points from the two classes. When the data points cannot be separated in the original space $\mathbb{R}^{n_d}$, we can search in a feature space $\mathbb{R}^{n_f}$, by mapping the input data space $\mathbb{R}^{n_d}$ to the feature space through a function $\varphi: \mathbb{R}^{n_d} \to \mathbb{R}^{n_f}$. Using the function $\varphi$, we can define a kernel function $\kappa: \mathbb{R}^{n_d} \times \mathbb{R}^{n_d}\to \mathbb{R}$ as $\kappa(a_j, a_{j'}):= \langle \varphi (a_j), \varphi(a_{j'}) \rangle$ for any $a_j, a_{j'} \in \mathbb{R}^{n_d}$, where $\langle \cdot, \cdot \rangle$ denotes the inner product. Popular choices of kernel functions in the SVM literature include the linear kernel function $\kappa_\text{LIN}$, the polynomial kernel function $\kappa_\text{POL}$, and the Gaussian kernel function $\kappa_\text{GAU}$: 
\begin{align*}
    & \kappa_\text{LIN} (a_j, a_{j'}) = a_j^T a_{j'}, \quad \forall a_j, a_{j'} \in \mathbb{R}^{n_d} \\
    & \kappa_\text{POL} (a_j, a_{j'}) = (1 + a_j^T a_{j'})^r, \quad r>0, \; \forall a_j, a_{j'} \in \mathbb{R}^{n_d}\\
    & \kappa_\text{GAU} (a_j, a_{j'}) = \exp^{-\frac{\|a_j -  a_{j'}\|^2}{2\sigma^2}}, \quad \sigma > 0,  \; \forall a_j, a_{j'} \in \mathbb{R}^{n_d}.
\end{align*}
We separate the given set $\mathcal{S}$ into a training set, $\mathcal{S}_\text{tr} = \{(a_{j},y_{j})\}_{j=1}^{N_\text{tr}}$ and a testing set,  $\mathcal{S}_\text{te} = \{(a_{j},y_{j})\}_{j=1}^{N_\text{te}}$, such that $N_\text{tr} + N_\text{te} = N$. Choosing a set of kernel functions $(\kappa_i)_{i=1}^m$, the SVM classifier is learned by solving the following convex QCQP problem on the training set  $\mathcal{S}_\text{tr}$:
\begin{align}\label{MuliSVM}
    \min_{ \alpha \in \mathbb{R}_{+}^{N_\text{tr}}, \;  \sum_{j=1}^{N_\text{tr}} y_{j} \alpha_{j} =0, \;  d \in \mathbb{R}} & \dfrac{1}{2C}  \|\alpha\|^2 - e^{T}\alpha + R d \\
   \text{s.t.} \quad & \dfrac{1}{2R_i} \alpha^{T}  (G_{i}(K_{i,\text{tr}})) \alpha - d \leq 0 \quad \forall i=1:m, \nonumber
\end{align}
where the parameter $C > 0$ is taken from the soft margin criteria, the vector $e$ denotes the $N_\text{tr}$-dimensional vector of all ones. {Given a labeled training data set \( \mathcal{S}_{\text{tr}} = \{(\mathbf{a}_j, l_j)\}_{j=1}^{N_{\text{tr}}} \) and an unlabeled test data set \( \mathcal{S}_{\text{te}} = \{\mathbf{a}_j\}_{j=1}^{N_{\text{te}}} \), a matrix \( K_i \in \mathbb{R}^{(N_{\text{tr}}+N_{\text{te}}) \times (N_{\text{tr}}+N_{\text{te}})} \) can be defined on the entire data set \( \mathcal{S}_{\text{tr}} \cup \mathcal{S}_{\text{te}} \) as follows:
$$K_i = \begin{pmatrix} K_{i,\text{tr}} & K_{i,(\text{tr,te})} \\ K_{i,(\text{tr,te})}^\top & K_{i,\text{te}} \end{pmatrix}.$$
The submatrix $K_{i, \text{tr}} \in \mathbb{R}^{N_\text{tr} \times N_\text{tr}}$ is a positive semidefinite matrix, whose $jj'$th element is defined by the kernel function: $[K_{i,\text{tr}}]_{jj'}:= \kappa_i(a_j, a_{j'}) $ for any $a_j, a_{j'} \in \mathcal{S}_\text{tr}$. The submatrices $K_{i, (\text{tr,te})} \in \mathbb{R}^{N_\text{tr} \times N_\text{te}}$ and $K_{i, \text{te}} \in \mathbb{R}^{N_\text{te} \times N_\text{te}}$ are defined in the same way but with different input vectors.}
The matrix $G_{i}(K_{i, \text{tr}}) \in \mathbb{R}^{N_\text{tr} \times N_\text{tr}}$ in the quadratic constraints of \eqref{MuliSVM} is a positive semidefinite  matrix with its $jj'$th element being $[G_i(K_{i,\text{tr}})]_{jj'} = y_j y_{j'} [K_{i,\text{tr}}]_{jj'}$. Moreover, $R$ and $R_i$'s are fixed positive constants. Clearly,  \eqref{MuliSVM} fits into the problem  \eqref{qcqp}, where: 
\begin{align*}
   &  x = 
\begin{bmatrix}
    \alpha \\ d
\end{bmatrix} \in \mathbb{R}^{N_\text{tr}+1}, 
\mathcal{Y} = 
\begin{bmatrix}
    \{\alpha \in \mathbb{R}_{+}^{N_\text{tr}}: \sum_{j=1}^{N_\text{tr}} y_{j} \alpha_{j} =0\}\\
    d \in \mathbb{R}
\end{bmatrix}, 
q_f = 
\begin{bmatrix}
    - e \\ R
\end{bmatrix} \in \mathbb{R}^{N_\text{tr}+1} \\
& Q_f = \frac{1}{C}
\begin{bmatrix}
    I_{N_\text{tr}}& 0\\
    0 & 0
\end{bmatrix}\in \mathbb{R}^{N_\text{tr}+1 \times N_\text{tr}+1},
Q_i = 
\begin{bmatrix}
    G_{i}(K_{i, \text{tr}})& 0\\
    0 & 0
\end{bmatrix}\in \mathbb{R}^{N_\text{tr}+1 \times N_\text{tr}+1}\\
& q_i = 
\begin{bmatrix}
    \mathbf{0} \\ \red{- R_i}
\end{bmatrix},
b_i = \mathbf{0}\in \mathbb{R}^{N_\text{tr}+1},
\end{align*}
with $I_{N_\text{tr}}$ the identity matrix and $\mathbf{0}$  the vector of zeros of appropriate dimensions. Since $Q_f$ and $Q_i$'s are positive semidefinite,  we can use the SMBA algorithm to solve the convex QCQP problem \eqref{MuliSVM} similar as we have done in Section $5.1$. Note that here the projection step \eqref{eq:alg2stepSMBA} in SMBA can be computed very efficiently using e.g., the algorithm in \cite{Kiw:08}. Once an optimal solution $(\alpha^*,d^*)$ is found, combining it with the associated multipliers corresponding to the quadratic constraints,  $(\lambda_i^*)_{i=1}^m$, and the pre-specified kernel functions $(\kappa_i)_{i=1}^m$, they can be used to label the test data set according to the following discriminant function:
\begin{align}\label{descri_func}
    \text{sign} \left(  \sum_{j=1}^{N_\text{tr}} y_j \alpha_j^* \left( \sum_{i=1}^m \lambda_i^* \kappa_i(a_j,a)  \right) + d^* \right) \quad \forall a \in \mathcal{S}_\text{te}.
\end{align} 


\noindent {The Test Set Accuracy (TSA) can then be obtained by measuring the percentage of the test data points accurately labeled according to the discriminant function \eqref{descri_func}}. In our experiments, we utilized real datasets sourced from the UCI Machine Learning Repository (\url{https://archive.ics.uci.edu/datasets}). Each dataset was divided into a training set comprising $80\%$ of the data and a testing set comprising the remaining $20\%$. We configured all parameters for the Stochastic Moving Ball Approximation (SMBA) algorithm according to the specifications outlined in Section $5.1$ for the convex case, including the defined stopping criteria. Furthermore, we employed a predefined set of Gaussian kernel functions \red{$(\kappa_i)_{i=1}^m$}, with the corresponding $\sigma^2_i$ values set to $m$ different grid points within the interval $[10^{-4}, 10^4]$. Following the pre-processing strategy outlined in \cite{CheLi:21}, we normalized each matrix {$K_{i}$} such that $R_i= \text{trace} ({K_{i}})$ was set to $1$, thus restricting $R = m$. For each dataset, we considered two different values for the number $m$ of grid points, namely $m = 10$ and $m = 50$.  Additionally, we set $C = 0.1$. Furthermore, in order to give a better overview of the advantages offered by the multiple kernel SVM approach,  we also learn a  single Gaussian kernel SVM classifier with $\sigma^2$  fixed a priori to $1$, by solving the following QP problem: 
\begin{align}\label{s-SVM}
    \min_{ 0 \leq \alpha \leq C, \;\; \sum_{j=1}^{N_\text{tr}} y_{j} \alpha_{j} =0 } & \frac{1}{2} \alpha^T G(K_{\text{tr}}) \alpha  - e^{T}\alpha. 
\end{align}

\medskip

\noindent \red{Since problem \eqref{s-SVM} does not have functional inequality constraints, we solve it with FICO.} Table \ref{table2} presents a comparison between SMBA and FICO in terms of CPU time for solving QCQP \eqref{MuliSVM} on real training sets: \texttt{Raisin}, \texttt{AIDS}, \texttt{Tuandromd}, \texttt{Predict students} and \texttt{Support2}. Notably, we did not compare with MBA \red{and MBA-AS} due to the consistent utilization of an infeasible initial point throughout the experiments. For each dataset, we also provided the nonzero optimal dual multiplier value corresponding to the active quadratic inequality constraint and the corresponding value of $\sigma^2$. Additionally, the last two columns of the table present a comparison between the Testing Set Accuracies (TSA) on the remaining testing datasets obtained by the multiple Gaussian kernel SVM classifier with $\sigma^2$ derived from  \eqref{MuliSVM}   and the single Gaussian kernel SVM classifier with $\sigma^2$  fixed a priori to $1$ in \eqref{s-SVM}, \red{denoted as TSA2}. \red{Note that while solving \eqref{MuliSVM} usually only one quadratic constraint is active at optimality (see  \cite{CheLi:21}) and thus the corresponding  nonzero dual multiplier for  SMBA is computed as follows: for the given primal solution, we select the most violated quadratic constraint and this is our choice for the active constraint; then, we compute the corresponding nonzero dual multiplier by solving a linear least-square problem arising from the optimality KKT condition. In the numerical experiments, we noticed that {for both algorithms, SMBA and FICO, we obtain the same index $i \in [1:m]$ corresponding to the nonzero dual multiplier and, moreover,} the values of the nonzero dual multiplier corresponding to the active inequality constraint are very close (see Table 2).  Hence, the values of $\sigma^2$ produced by FICO and SMBA coincide. {Furthermore, the stopping criteria used in SMBA guarantee that the optimal points $(\alpha^*, d^*)$ obtained by SMBA are very close to the optimal points $(\alpha^*, d^*)$ obtained by FICO.} Therefore, when computing the TSA based on the solutions given by SMBA and FICO, respectively, we get the same values. Hence, in Table 2 we report only once the values for  TSA.} The results indicate that SMBA generally outperforms FICO in terms of training CPU time, and the TSA obtained by solving \eqref{MuliSVM} is consistently superior to TSA2 obtained by solving \eqref{s-SVM}.

\begin{table}[ht]
\centering
\caption{\noindent Comparison of SMBA ($\beta = 0.96$), and FICO solver in terms of CPU time (in sec.) to solve \eqref{MuliSVM} for various real data and two different $m = 10, 50$. Additionally, the TSA for \eqref{MuliSVM} \red{based on the solution provided by SMBA} and \red{TSA2 for \eqref{s-SVM} based on the solution provided by FICO}.}
\label{table2}
\begin{tabular}{|l|l|ll|ll|l|l|}
\hline
\multicolumn{1}{|c|}{\multirow{2}{*}{\begin{tabular}[c]{@{}c@{}}Dataset\\ $(N, n_d)$\end{tabular}}} & \multicolumn{1}{c|}{\multirow{2}{*}{m}} & \multicolumn{2}{c|}{SMBA  $(\beta = 0.96)$}                                     & \multicolumn{2}{c|}{FICO} & \multirow{2}{*}{TSA} & \multicolumn{1}{c|}{\multirow{2}{*}{TSA2}} \\ 
\multicolumn{1}{|c|}{}                                                                              & \multicolumn{1}{c|}{}                   & \multicolumn{1}{c|}{Time} & \multicolumn{1}{c|}{$\sigma^2$, Nonzero $\lambda_i^*$} & \multicolumn{1}{c|}{Time}  & \multicolumn{1}{c|}{Nonzero $\lambda_i^*$}  &                      & \multicolumn{1}{c|}{}                      \\ \hline

\multirow{2}{*}{\begin{tabular}[c]{@{}c@{}}\texttt{Raisin}\\ $(900, 7)$\end{tabular}}        &   10     & \multicolumn{1}{c|}{\textbf{0.13}}  &   \multicolumn{1}{c|}{$1.11 \times 10^3$, $\lambda_2^* = 0.22$}         &   \multicolumn{1}{c|}{0.39}    &   \multicolumn{1}{c|}{\red{$\lambda_2^* = 0.219$}}      &    \multicolumn{1}{c|}{71.66}   &        \multirow{2}{*}{69.11}       \\ 
&      50           & \multicolumn{1}{c|}{\textbf{0.35}}  &    \multicolumn{1}{c|}{$204.08$, $\quad \;\;\lambda_2^* = 0.22$}        &    \multicolumn{1}{c|}{1.72}   &   \multicolumn{1}{c|}{\red{$\lambda_2^* = 0.218$}}     &    \multicolumn{1}{c|}{82.22}   &     \\ \hline

\multirow{2}{*}{\begin{tabular}[c]{@{}c@{}}\texttt{AIDS}\\ $(2139, 23)$\end{tabular}}       &   10     & \multicolumn{1}{c|}{\textbf{0.87}}  &   \multicolumn{1}{c|}{$5.55 \times 10^3$, $\lambda_6^* = 1.62$}         &   \multicolumn{1}{c|}{0.98}   &   \multicolumn{1}{c|}{\red{$\lambda_6^* = 1.61$}}       &    \multicolumn{1}{c|}{77.28}   &        \multirow{2}{*}{52.71}       \\ 
&      50           & \multicolumn{1}{c|}{\textbf{2.53}}  &    \multicolumn{1}{c|}{$2.44\times 10^3$, $\;\lambda_{13}^* = 1.61$}        &    \multicolumn{1}{c|}{5.41}   &   \multicolumn{1}{c|}{\red{$\lambda_{13}^* = 1.60$}}     &    \multicolumn{1}{c|}{80.17}   &     \\ \hline

\multirow{2}{*}{\begin{tabular}[c]{@{}c@{}}\texttt{Tuandromd}\\ $(4465, 242)$\end{tabular}}       &   10     & \multicolumn{1}{c|}{18.13}  &   \multicolumn{1}{c|}{$1\times 10^{-4}$, $\lambda_1^* = 0.5$}         &   \multicolumn{1}{c|}{\textbf{16.34}}     &   \multicolumn{1}{c|}{\red{$\lambda_1^* = 0.49$}}     &    \multicolumn{1}{c|}{53.47}   &        \multirow{2}{*}{47.19}       \\ 
&      50           & \multicolumn{1}{c|}{\textbf{17.12}}  &    \multicolumn{1}{c|}{$1\times 10^{-4}$, $\;\lambda_1^* = 0.02$}        &    \multicolumn{1}{c|}{77.25}    &   \multicolumn{1}{c|}{\red{$\lambda_1^* = 0.018$}}    &    \multicolumn{1}{c|}{55.49}   &     \\ \hline

\multirow{2}{*}{\begin{tabular}[c]{@{}c@{}}\texttt{Predict stu.}\\ $(4424, 36)$\end{tabular}}    &   10     & \multicolumn{1}{c|}{\textbf{15.63}}  &   \multicolumn{1}{c|}{$1\times 10^{-4}$, $\;\lambda_1^* = 0.05$}         &   \multicolumn{1}{c|}{16.65}    &   \multicolumn{1}{c|}{\red{$\lambda_1^* = 0.051$}}     &    \multicolumn{1}{c|}{67.55}   &        \multirow{2}{*}{59.11}       \\ 
&      50           & \multicolumn{1}{c|}{84.24}  &    \multicolumn{1}{c|}{$1\times 10^{-4}$, $\lambda_1^* = 0.51$}        &    \multicolumn{1}{c|}{\textbf{77.59}}    &   \multicolumn{1}{c|}{\red{$\lambda_1^* = 0.52$}}   &    \multicolumn{1}{c|}{72.67}   &     \\ \hline

\multirow{2}{*}{\begin{tabular}[c]{@{}c@{}}\texttt{Support2}\\ $(9105, 42)$\end{tabular}}    &   10     & \multicolumn{1}{c|}{\textbf{28.34}}  &   \multicolumn{1}{c|}{$4.44 \times 10^3$, $\lambda_5^* = 0.25$}         &   \multicolumn{1}{c|}{31.44}    &   \multicolumn{1}{c|}{\red{$\lambda_5^* = 0.22$}}      &    \multicolumn{1}{c|}{85.27}   &        \multirow{2}{*}{74.28}       \\ 
&      50           & \multicolumn{1}{c|}{\textbf{57.81}}  &    \multicolumn{1}{c|}{$204.08$, $\quad \;\;\lambda_2^* = 0.51$}        &    \multicolumn{1}{c|}{162.59}    &   \multicolumn{1}{c|}{\red{$\lambda_2^* = 0.49$}}    &    \multicolumn{1}{c|}{87.05}   &     \\ \hline
\end{tabular}
\tiny{\textbf{Note:} TSA = Testing Set Accuracy (in $\%$) for \eqref{MuliSVM}, TSA2 = Testing Set Accuracy (in $\%$) for \eqref{s-SVM}.}
\end{table}


\medskip 

\noindent From the preliminary numerical experiments on QCQPs with synthetic and real data matrices we can conclude that our algorithm SMBA is a better and viable alternative to some existing state-of-the-art methods and software.


\section{Conclusions} 
In this paper, we have considered an optimization problem with a smooth objective being either convex or strongly convex and with a large number of smooth convex functional constraints. We have designed a new primal stochastic gradient type algorithm, called Stochastic Moving Ball Approximation (SMBA), to solve such optimization problems where at each iteration SMBA performs a gradient step for the objective function to generate an intermittent point, and then takes a projection step onto one ball approximation of a randomly chosen constraint. The computational simplicity of the SMBA subproblem makes it suitable for problems with many functional constraints. We have provided a convergence analysis for our algorithm, deriving convergence rates of order   $\mathcal{O} (k^{-1/2})$,  when the objective function is convex, and $\mathcal{O} (k^{-1})$, when the objective function is strongly convex. Finally, we apply  SMBA algorithm to solve QCQPs based on synthetic and real data, demonstrating its viability and performance when compared to MBA  and two commercial softwares Gurobi~and~FICO.


\begin{appendices}
\section{}
\label{secA1}
\textit{Equivalent form of the feasibility step in SMBA algorithm.} In this appendix we prove that if $R_{v_k,\xi_k} > 0$ and $h(v_k,\xi_k) > 0$, then the step \eqref{eq:alg2step3SMBA} in SMBA algorithm can be written equivalently as:
\begin{align*}
    v_k - \frac{\beta}{L_{\xi_k}}\left(1 - \frac{\sqrt{R_{{v}_k,\xi_k}} }{\|v_k - c_{{v}_k,\xi_k}\|} \right) \nabla h (v_k, \xi_k) = (1-\beta)v_k + \beta  \Pi_{\mathcal{B}_{{v}_k,\xi_k}} (v_k).
\end{align*}
Indeed, if $R_{{v}_k,\xi_k} > 0$, we can compute the projection of $v_k$ onto the nonempty ball:
    \[ \mathcal{B}_{{v}_k,\xi_k} = \left\{y: \; q_{h(\cdot,\xi_k)}(y;{v}_k) \le 0 \right\} = \left\{y: \; \|y-c_{{v}_k,\xi_k}\|^2 \le R_{{v}_k,\xi_k} \right\},   \]
which has the explicit expression: 
\[ \Pi_{\mathcal{B}_{{v}_k,\xi_k}} (v_k)  = c_{{v}_k,\xi_k} + \frac{\sqrt{R_{{v}_k,\xi_k}}}{\max\left(\|v_k - c_{{v}_k,\xi_k}\|, \sqrt{R_{{v}_k,\xi_k}}\right) }(v_k - c_{{v}_k,\xi_k}), \]
    or, after adding and subtracting $v_k$, we obtain:
    \begin{align*}
       \Pi_{\mathcal{B}_{{v}_k,\xi_k}} (v_k)  = v_k - \left(1 - \frac{\sqrt{R_{{v}_k,\xi_k}} }{\max\left(\|v_k - c_{{v}_k,\xi_k}\|, \sqrt{R_{{v}_k,\xi_k}}\right)} \right) (v_k - c_{{v}_k,\xi_k}). 
    \end{align*}
Note that when $h(v_k,\xi_k) >0$ and $R_{{v}_k,\xi_k}>0$ we have the following (see \eqref{eq:mid1SMBA}): $$\max\left(\|v_k - c_{{v}_k,\xi_k}\|, \sqrt{R_{{v}_k,\xi_k}}\right) = \|v_k - c_{{v}_k,\xi_k}\|.$$ 
Thus, in this case, we have:
\begin{align*}
    (1-\beta)v_k + \beta  \Pi_{\mathcal{B}_{{v}_k,\xi_k}} (v_k) & = v_k - \beta\left(1 - \frac{\sqrt{R_{{v}_k,\xi_k}} }{\|v_k - c_{{v}_k,\xi_k}\|} \right) (v_k - c_{{v}_k,\xi_k})\\
    & = v_k - \frac{\beta}{L_{\xi_k}}\left(1 - \frac{\sqrt{R_{{v}_k,\xi_k}} }{\|v_k - c_{{v}_k,\xi_k}\|} \right) \nabla h (v_k, \xi_k).
\end{align*}

\end{appendices}


\section*{Acknowledgements}
\noindent The research leading to these results has received funding from UEFISCDI PN-III-P4-PCE-2021-0720, under project L2O-MOC, nr. 70/2022.\\

\noindent \textbf{Data availability} Not applicable.\\

\noindent\textbf{Conflicts of interest} The authors declare that they have no conflict of interest.


\end{document}